\documentclass[12pt,english,a4paper,oneside]{amsart}

\usepackage[cp1251]{inputenc}
\usepackage[english]{babel}
\usepackage{amsmath,amsthm,amssymb}

\usepackage[
a4paper,
left=22mm,
right=18mm,
top=20mm,
bottom=27mm
]{geometry}

\usepackage{xcolor}
\usepackage[colorlinks]{hyperref}

\newtheorem{theorem}{Theorem}[section]
\newtheorem{lemma}[theorem]{Lemma}
\newtheorem{corollary}[theorem]{Corollary}

\numberwithin{equation}{section}
\theoremstyle{definition}
\newtheorem{definition}[theorem]{Definition}
\newtheorem{remark}[theorem]{Remark}
\newtheorem{example}[theorem]{Example}

\begin{document}

\title{Nondegeneracy and regularity of polynomial pushforwards}

\author{Egor Kosov}

\address{\noindent Egor Kosov,
Centre de Recerca Matem\`atica, Campus de Bellaterra, Edifici~C 08193
Bellaterra (Barcelona), Spain.}
\email{kosoved09@gmail.com}

\author{Anastasiia Zhukova}

\address{\noindent
Anastasiia Zhukova,
Faculty of Mechanics and Mathematics, Lomonosov Moscow State University, Moscow, 119991 Russia}



\begin{abstract}
Let $\mu$ be a log-concave probability measure on $\mathbb R^n$ and let
$f\colon\mathbb R^n\to\mathbb R^k$ be a polynomial mapping of degree at most
$d$. We show that
\[
\mu(f\in A)
\le
C\bigl(\lambda_k(A)\bigr)^{\frac{1}{k(d-1)+1}}
\]
for every Borel set $A\subset\mathbb R^k$
whenever the image measure $\mu\circ f^{-1}$ is absolutely
continuous.
The constant $C$ is independent of the dimension $n$, and the
exponent $\frac{1}{k(d-1)+1}$ is sharp. This extends the scalar
Carbery--Wright inequality and answers, in the log-concave setting, a question raised by Avni, Glazer, and Larsen.
In addition, we show that the density of $\mu\circ f^{-1}$, whenever it
exists, belongs to the Nikolskii--Besov space
$B^{\frac{1}{k(d-1)+1}}_{1,\infty}(\mathbb R^k)$,
with a dimension-free bound for the corresponding norm.

A central difficulty in passing from scalar polynomials to vector-valued
polynomial mappings is the lack of a suitable nondegeneracy parameter
quantifying absolute continuity of $\mu\circ f^{-1}$, as the variance does in
the scalar case. Natural candidates such as the covariance matrix or
the Jacobian matrix either fail to characterize this property or do not lead
to dimension-free estimates. 
We identify such a parameter and define it to be the covariance matrix of the
vector formed by the monomials of degree up to $d^{k-1}$ in the normalized
components of $f$.

The dimension-free nature of our results allows us to extend Kusuoka's
absolute continuity criterion for Gaussian polynomial random vectors to the
log-concave setting. Moreover, in this setting, we obtain estimates relating
convergence in distribution to convergence in total variation for polynomial
random vectors.
\end{abstract}
\maketitle

\section{Introduction}

It is by now well understood that scalar polynomials on convex bodies and,
more generally, with respect to log-concave measures possess a number of
important dimension-free analytic and distributional properties. These include
moment comparison inequalities, sharp small-ball estimates, quantitative
decay bounds for oscillatory integrals with polynomial phases, and fractional
regularity estimates for the densities of the associated one-dimensional image
measures. Such properties were studied, among others, by Bourgain
\cite{Bourgain91}, Bobkov \cite{Bobkov00}, Carbery and Wright~\cite{CW01},
and Nazarov, Sodin, and Volberg \cite{NSV2003}.

The main goal of the present paper is to go beyond the scalar
setting and to develop a dimension-free theory for image measures induced by
vector-valued polynomial mappings.
We denote by
$\mathcal P_d(\mathbb R^n)$ the space of all algebraic polynomials of degree
at most $d$ on $\mathbb R^n$, and by
$\mathcal P_d(\mathbb R^n;\mathbb R^k)$ the space of all mappings
$f=(f_1,\ldots,f_k)$ such that $f_j\in\mathcal P_d(\mathbb R^n)$ for all
$j=1,\ldots,k$.

We consider arbitrary log-concave
measures $\mu$ on $\mathbb R^n$ and mappings
$f\in\mathcal P_d(\mathbb R^n;\mathbb R^k)$, and investigate quantitative
properties of the image measure $\mu\circ f^{-1}$. We are interested in
small-ball estimates for $\mu\circ f^{-1}$ and in integrability and fractional
regularity properties of its density, when this density exists. The focus is on
dimension-free estimates, with constants independent of the ambient
dimension $n$. 
This viewpoint is motivated, in particular, by applications to probability and
to limit theorems for polynomial random vectors discussed below.

\subsection{Small-ball estimates}

The small-ball problem asks for the rate at which the concentration function
\begin{equation}\label{small-ball-def}
L(\mu\circ f^{-1},t):=
\sup_{y\in\mathbb R^k}\mu(|f-y|\le t)
=
\sup_{y\in\mathbb R^k}\mu\bigl(f\in B_t(y)\bigr)
\end{equation}
decays as $t\downarrow0$, where $|\cdot|$ denotes the Euclidean norm on
$\mathbb R^k$ and $B_t(y)$ is the Euclidean ball of radius $t$ centered at
$y$. 
Replacing balls in \eqref{small-ball-def} by arbitrary Borel sets leads to a
stronger form of the small-ball problem: to control
$\mu(f\in A)$
in terms of the Lebesgue measure $\lambda_k(A)$, uniformly over all Borel sets
$A\subset\mathbb R^k$. More precisely, one seeks estimates of the form
\begin{equation}\label{lorentz}
	\mu(f\in A)\le C\bigl(\lambda_k(A)\bigr)^\alpha
\end{equation}
for all Borel sets $A\subset\mathbb R^k$, with some $0<\alpha\le 1$.

For sums of independent random vectors, which correspond to linear images of
product measures generated by mappings
$f\in\mathcal P_1(\mathbb R^n;\mathbb R^k)$, these questions go back to the
works of L\'evy~\cite{Levy37} and Kolmogorov~\cite{Kolmogorov58}. More recently, they were studied by Rudelson and
Vershynin~\cite{RV15} and by Bobkov and Chistyakov~\cite{BCh14} in the
endpoint case $\alpha=1$. See also \cite{LPP16,MMX17,DPP18,Kos22} for further
developments.

In the scalar polynomial case, such estimates have been studied, in
particular, by Phong, Stein, and Sturm \cite{PSS01}, Carbery, Christ, and
Wright \cite{CCW99}, and, in a more general geometric form, by
Gressman~\cite{Gressman11}. When $\mu$ is log-concave, Carbery and Wright
\cite{CW01} and, independently, Nazarov, Sodin, and Volberg \cite{NSV2003}
proved that, for $f\in\mathcal P_d(\mathbb R^n)$, one has
\begin{equation}\label{CW-est}
	\|f\|_{L^2(\mu)}^{1/d}\mu(|f|\le t)
	\le C d\, t^{1/d}
	\quad \forall t>0,
\end{equation}
where $C>0$ is a universal constant independent of the dimension.
Consequently,
\begin{equation}\label{small-ball-CW}
L(\mu\circ f^{-1},t)\le
Cd(\operatorname{Var}_\mu(f))^{-\frac{1}{2d}}
\, t^{1/d}
\quad \forall t>0
\end{equation}
provided that the variance $\operatorname{Var}_\mu(f)$ of $f$ with respect
to $\mu$ is positive.
Moreover, the results of \cite{Kos25,Kos18} imply that, in the scalar case
$f\in\mathcal P_d(\mathbb R^n)$, one actually has the stronger estimate
\begin{equation}\label{Kos-est}
	\mu(f\in A)\le
	Cd(\operatorname{Var}_\mu(f))^{-\frac{1}{2d}}
	\bigl(\lambda_1(A)\bigr)^{1/d}
\end{equation}
for all Borel sets $A\subset\mathbb R$.
In particular, in this setting, the absolute continuity of
$\mu\circ f^{-1}$ is equivalent to the nondegeneracy condition
$\operatorname{Var}_\mu(f)>0$. Thus, \eqref{small-ball-CW} and \eqref{Kos-est}
hold if and only if the image measure $\mu\circ f^{-1}$ is absolutely
continuous.

In this paper, we obtain sharp dimension-free estimates of the form \eqref{lorentz} for
general polynomial mappings
$f\in\mathcal P_d(\mathbb R^n;\mathbb R^k)$ and log-concave measures $\mu$.
Namely, whenever the image measure
$\mu\circ f^{-1}$ is absolutely continuous, we prove that
\begin{equation}\label{lorentz-our}
	\mu(f\in A)
	\le
	C
	\bigl(\lambda_k(A)\bigr)^{\frac{1}{k(d-1)+1}}
\end{equation}
for every Borel set $A\subset\mathbb R^k$, 
where the constant $C$ depends on $k$, $d$, and certain quantitative
nondegeneracy parameters of $f$ with respect to $\mu$, but not on the ambient
dimension $n$.
This gives a vector-valued counterpart of the scalar estimate
\eqref{Kos-est} and implies the corresponding small-ball bound, thereby
extending the scalar Carbery--Wright estimate \eqref{CW-est}.
In particular, for log-concave measures, this answers the dimension-free
version of the question raised by Avni, Glazer, and Larsen in
\cite[Problem~1.12(2)]{AGL24} and the discussion following it.

To the best of our knowledge, prior to the present work, estimates of the form
\eqref{lorentz} for nonlinear
vector-valued polynomial mappings were available only in the case $n=k$. Such an estimate is obtained as a key part of
the proof of \cite[Theorem~4.1]{CRW03} by Carbery, Ricci, and Wright, in
their study of polynomial maximal operators. Namely, for
$f\in\mathcal P_d(\mathbb R^k;\mathbb R^k)$ and every Borel set
$A\subset\mathbb R^k$, they proved
\begin{equation}\label{CRW-est}
\mu_B(f\in A)
\le
C(k,d)
\biggl(\int_B |\det J_f(x)|\,dx\biggr)^{-\frac{1}{k(d-1)+1}}
\bigl(\lambda_k(A)\bigr)^{\frac{1}{k(d-1)+1}},
\end{equation}
where $B\subset\mathbb R^k$ is a Euclidean ball, $\mu_B$ is the normalized
Lebesgue measure on $B$, and $J_f$ denotes the Jacobian matrix of $f$.

\subsection{Algebraic nondegeneracy parameter}

In the scalar case, the small-ball estimate \eqref{small-ball-CW} and its
generalization \eqref{Kos-est} are governed by a simple nondegeneracy
parameter of the polynomial, namely its variance $\operatorname{Var}_\mu(f)$.
The assumption $\operatorname{Var}_\mu(f)>0$ is precisely what ensures that
the image measure $\mu\circ f^{-1}$ is absolutely continuous. In the vector-valued case, one of the main conceptual obstructions
is that the variance has no immediate analogue capable of quantifying
nondegeneracy of the image measure $\mu\circ f^{-1}$.

A natural first candidate is the covariance matrix $\operatorname{Cov}_\mu(f)$
of the vector $f$ with respect to $\mu$.
However, a simple example shows that nondegeneracy of this covariance matrix
is not sufficient to guarantee absolute continuity of $\mu\circ f^{-1}$ for a
general vector-valued polynomial mapping.

\begin{example}\label{example-1}
Consider the mapping $f=(f_1,f_2,f_3)$ on $[-\frac12,\frac12]^2$, equipped
with the uniform measure $\mu$, where
\[
f_1(x,y):=x^2-\tfrac{1}{12},\quad
f_2(x,y):=y^2-\tfrac{1}{12},\quad
f_3(x,y):=xy.
\]
Then $\operatorname{Cov}_\mu(f)$ is nondegenerate, while the image measure
$\mu\circ f^{-1}$ is singular, since
\[
\bigl(f_1+\tfrac{1}{12}\bigr)\bigl(f_2+\tfrac{1}{12}\bigr)-f_3^2=0.
\]
\end{example}

Another natural attempt to quantify nondegeneracy of the image measure
$\mu\circ f^{-1}$, already appearing in \eqref{CRW-est}, is based on the Jacobian matrix $J_f$ of the mapping $f$.
However, such a derivative-based parameter cannot lead to dimension-free
results for arbitrary log-concave measures, even under isotropic normalization
and already in the scalar case.

\begin{example}\label{example-3}
Let $\mu_n$ be the normalized restriction of the Lebesgue measure to the
Euclidean ball
\[
B^n_2:=\bigl\{x\in\mathbb R^n\colon |x|\le \sqrt{n+2}\bigr\}.
\]
Then $\mu_n$ is isotropic.
Let
$g(x):= |x|^2 - n$.
It is readily seen that
\[
\mu_n(|g|\le 1)
\ge
\mu_n\bigl(\sqrt n\le |x|\le \sqrt{n+1}\bigr)
=
\frac{(n+1)^{n/2}-n^{n/2}}{(n+2)^{n/2}}
\ge
(2e)^{-1}
\]
for all $n\ge2$. 
Since $\|\nabla g\|_{L^2(\mu_n)}^2=4n$, we obtain
\[
\|\nabla g\|_{L^2(\mu_n)}^{1/2}\mu_n(|g|\le 1)
\ge
(4n)^{1/4}(2e)^{-1}\to\infty
\quad\text{as}\quad n\to\infty.
\]
Hence, there is no dimension-free analogue of
\eqref{small-ball-CW} with the variance replaced by
$\|\nabla f\|_{L^2(\mu_n)}^2$.
\end{example}

Examples~\ref{example-1} and~\ref{example-3} show that, in the
vector-valued setting, the desired dimension-free results cannot be governed
either by the covariance matrix of $f$ itself or by the Jacobian matrix~$J_f$.
Motivated by these obstructions, we introduce the following new
nondegeneracy parameter, which detects all polynomial relations that can
occur among the components of a mapping
$f\in\mathcal P_d(\mathbb R^n;\mathbb R^k)$ and
quantifies the nondegeneracy of $\mu\circ f^{-1}$.

\begin{definition}[Extended mapping]\label{def-ext}
Let $\mu$ be a Borel probability measure on $\mathbb R^n$, let $d\in\mathbb N$,
and let
$f=(f_1,\ldots,f_k)\colon \mathbb R^n\to \mathbb R^k$ be a measurable mapping
such that
\[
0<\sigma_{f_j}^2:=\operatorname{Var}_\mu(f_j)<\infty
\quad \forall j\in\{1,\ldots,k\}.
\]
Let
\[
\hat f_j
=
\sigma_{f_j}^{-1}
\biggl(f_j-\int_{\mathbb R^n} f_j\,d\mu\biggr)
\quad \forall j\in\{1,\ldots,k\},
\]
and set
$\hat f:=(\hat f_1,\ldots,\hat f_k)$.
Assume also that
\[
\hat f^{\mathbf m}:=\hat f_1^{m_1}\ldots \hat f_k^{m_k}\in L^2(\mu)
\quad \forall \mathbf m=(m_1,\ldots,m_k)\in\mathbb Z_+^k
\quad\text{with}\quad
1\le |\mathbf m|\le d^{k-1}.
\]
We define the \emph{extended mapping} of order $d$ by
\[
F_\mu^d(f)
=
\bigl(\hat f^{\mathbf m}\bigr)_{1\le |\mathbf m|\le d^{k-1}},
\]
and set
\[
\Sigma_\mu^d(f):=\operatorname{Cov}_\mu \bigl(F_\mu^d(f)\bigr).
\]
In other words, $\Sigma_\mu^d(f)$ is the matrix, in the monomial basis, of the quadratic form
\[
Q\mapsto \operatorname{Var}_\mu\bigl(Q(\hat f)\bigr),\quad Q\in\mathcal P_{d^{k-1}}(\mathbb R^k),\quad
Q(0)=0.
\]
\end{definition}

\smallskip

To explain the meaning of this definition, we note that, if the measure
$\mu\circ f^{-1}$ is absolutely continuous, then necessarily
\begin{equation}\label{assumptions}
\operatorname{Var}_{\mu}(f_j)>0
\quad \forall j\in\{1,\ldots,k\}
\quad
\text{and}
\quad
\det \Sigma_\mu^d(f)>0.
\end{equation}
Indeed, otherwise $\mu\circ f^{-1}$ would be concentrated on the zero set of
a nonzero algebraic polynomial in $k$ variables, which has Lebesgue measure
zero, see, for instance, \cite{Mityagin20}.

Conversely, the assumptions \eqref{assumptions} exclude nontrivial polynomial
dependence of the components of $f$ up to degree $d^{k-1}$.
A key point of the present paper is that this is sufficient
in our setting: for every mapping
$f\in\mathcal P_d(\mathbb R^n;\mathbb R^k)$ and every log-concave measure
$\mu$, conditions \eqref{assumptions} guarantee absolute continuity of
$\mu\circ f^{-1}$ and yield the desired dimension-free estimates.
We note that the degree cutoff $d^{k-1}$ arises naturally in this setting
from a Perron-type bound on the degree of an annihilating polynomial for
algebraically dependent polynomials, see
\cite{Ploski86,Ploski05}.

\subsection{The dimension-free set estimate}

We are now in a position to state our first main theorem in detail.

\begin{theorem}\label{th-main-lorentz}
Let $a, \eta>0$, $k,d\in\mathbb N$, and $d\ge2$. There exists
a constant $C:=C(a,\eta,k,d)>0$ such that, for every $n\in\mathbb N$, every
log-concave measure $\mu$ on $\mathbb R^n$, and every polynomial mapping
$
f=(f_1,\ldots,f_k)\in\mathcal P_d(\mathbb R^n;\mathbb R^k)
$
satisfying
\[
\operatorname{Var}_{\mu}(f_j)\ge a
\quad
\forall j\in\{1,\ldots,k\}
\quad
\text{and}
\quad
\det \Sigma_\mu^d(f)\ge \eta,
\]
one has
\[
\mu(f\in A)
\le
C\bigl(\lambda_k(A)\bigr)^{\frac{1}{k(d-1)+1}}
\]
for every Borel set $A\subset\mathbb R^k$.
\end{theorem}

\begin{remark}
In Theorem~\ref{th-main-lorentz}, one may take the constant in the form
\[
C(a,\eta,k,d)
=
C'(\eta,k,d)a^{-\frac{k}{2(k(d-1)+1)}}.
\]
\end{remark}

\medskip

In particular, the theorem gives a criterion for the absolute continuity of
$\mu\circ f^{-1}$ in the case of log-concave measures $\mu$ and mappings
$f\in \mathcal P_d(\mathbb R^n;\mathbb R^k)$. Namely, $\mu\circ f^{-1}$ is
absolutely continuous if and only if the conditions in \eqref{assumptions} are
satisfied. Moreover, applying the estimate of
Theorem~\ref{th-main-lorentz} to Euclidean balls gives the following
multidimensional counterpart of the Carbery--Wright inequality:
\[
L(\mu\circ f^{-1},t)
\le
C_3(a,\eta,k,d)\, t^{\frac{1}{d-1+\frac{1}{k}}}
\quad \forall t>0.
\]

The exponent $\frac{1}{k(d-1)+1}$ in Theorem~\ref{th-main-lorentz} is optimal for every $k$ and $d$, already for the uniform measure on $[0,1]^k$, as shown by the following example.

\begin{example}\label{example-2}
Consider the mapping $f\colon \mathbb R^k\to\mathbb R^k$ defined by
\[
f(x):=(x_1^d,x_1^{d-1}x_2,\ldots,x_1^{d-1}x_k).
\]

Let $\mu$ be the restriction of the Lebesgue measure on $[0,1]^k$. For $s\in(0,1)$, set
\[
A_s
:=
\bigl\{0\le y_1\le s^d,\ 0\le y_j\le y_1^{\frac{d-1}{d}},\ j=2,\ldots,k\bigr\}.
\]
Then
\[
B_s
:=
\{x\in[0,1]^k\colon f(x)\in A_s\}
=
\{0\le x_1\le s,\ x_2,\ldots,x_k\in[0,1]\}.
\]
Hence
$\mu(B_s)=s$
and
\[
\lambda_k(A_s)
=
\int_0^{s^d}y_1^{\frac{(k-1)(d-1)}{d}}\,dy_1
=
\frac{d}{k(d-1)+1}s^{k(d-1)+1}.
\]
Therefore, if an estimate of the form
\[
\mu(f\in A)\le C\bigl(\lambda_k(A)\bigr)^\alpha
\]
held for all Borel sets $A\subset\mathbb R^k$, then applying it to $A_s$ would give
\[
s
\le
C\bigl(\tfrac{d}{k(d-1)+1}\bigr)^\alpha
s^{\alpha(k(d-1)+1)}
\quad \forall s\in(0,1).
\]
Letting $s\downarrow0$, we necessarily obtain
\[
\alpha\le \tfrac{1}{k(d-1)+1}.
\]
Finally, for this mapping and this measure, the nondegeneracy assumptions of
Theorem~\ref{th-main-lorentz} are satisfied with some positive constants,
since $f([0,1]^k)$ has nonempty interior.
\end{example}

\subsection{Lorentz estimates for the density}

The estimate \eqref{lorentz} implies that $\mu\circ f^{-1}$ is absolutely
continuous and, for $0<\alpha<1$, that its density $\varrho_f$ belongs to the
weak Lorentz space $L^{p(\alpha),\infty}(\mathbb R^k)$, where
$p(\alpha)=\frac{1}{1-\alpha}$.
Moreover, for the standard weak-Lorentz quasi-norm
\[
\|\varrho_f\|_{L^{p,\infty}(\mathbb R^k)}
:=
\sup_{s>0}
s\,\bigl(\lambda_k(\varrho_f\ge s)\bigr)^{1/p},
\]
the best possible constant $C$ in \eqref{lorentz} is equivalent to
$\|\varrho_f\|_{L^{p(\alpha),\infty}(\mathbb R^k)}$ up to a constant depending
only on $\alpha$.
Therefore, Theorem~\ref{th-main-lorentz} implies that, under its assumptions,
the density of the measure $\mu\circ f^{-1}$ belongs to
$L^{\frac{k(d-1)+1}{k(d-1)},\infty}(\mathbb R^k)$.
Since this density is a probability density, it also belongs to
$L^q(\mathbb R^k)$ for every
$1<q<\frac{k(d-1)+1}{k(d-1)}$.

\subsection{Nikolskii--Besov regularity}

In the scalar case $f\in\mathcal P_d(\mathbb R^n)$, the sharp set estimate
\eqref{Kos-est} follows from a stronger Nikolskii--Besov regularity estimate
for the corresponding density~$\varrho_f$. Namely, it was proved in
\cite{Kos25,Kos18} that
\begin{equation}\label{eq-scalar-reg}
	\|\varrho_f\|_{\dot{B}^{1/d}_{1,\infty}(\mathbb R)}
	\le
	Cd(\operatorname{Var}_\mu(f))^{-\frac{1}{2d}}
\end{equation}
provided that $\operatorname{Var}_\mu(f)>0$. Here
$B^\alpha_{1,\infty}(\mathbb R^k)$, $\alpha\in(0,1)$, denotes the
Nikolskii--Besov space, see \cite{BIN,Stein}.

In dimension one, the standard Besov-to-Lorentz embedding,
see \cite[Corollary~4.20]{BennettSharpley88}, gives
\[
\|\varrho_f\|_{L^{p(\alpha),\infty}(\mathbb R)}
\le
C(\alpha)
\|\varrho_f\|_{\dot B^\alpha_{1,\infty}(\mathbb R)},
\quad
p(\alpha)=\frac{1}{1-\alpha}.
\]
Consequently,
\[
\mu(f\in A)
\le
C(\alpha)
\|\varrho_f\|_{\dot B^\alpha_{1,\infty}(\mathbb R)}
\bigl(\lambda_1(A)\bigr)^\alpha
\]
for every Borel set $A\subset\mathbb R$, and, in particular,
\[
L(\mu\circ f^{-1},t)
\le
C(\alpha)
\|\varrho_f\|_{\dot B^\alpha_{1,\infty}(\mathbb R)}
t^\alpha.
\]
Thus, in the scalar case, Nikolskii--Besov regularity implies Lorentz-space
and small-ball estimates of the same order, allowing one to study these
problems simultaneously. In addition, this regularity is connected with
problems in harmonic analysis concerning oscillatory integrals with polynomial
phases, since Nikolskii--Besov regularity of order $\alpha$ yields decay
estimates of the same order for such integrals. In particular, the estimate \eqref{eq-scalar-reg} settled the
Carbery--Wright conjecture \cite[Section~6]{CW01} on the sharp decay rate of
oscillatory integrals with polynomial phases over convex sets.

In the vector-valued case, the direct connection between Nikolskii--Besov
fractional regularity and estimates of the form \eqref{lorentz} breaks down.
Indeed, even in the scalar case, a general Nikolskii--Besov regularity theorem
for image measures generated by polynomial mappings of degree at most $d$
cannot have regularity order $\alpha$ larger than $1/d$.
Moreover, for densities on $\mathbb R^k$, the Besov-to-Lorentz embedding gives
only
\[
B^\alpha_{1,\infty}(\mathbb R^k)
\subset
L^{p(\alpha/k),\infty}(\mathbb R^k),
\quad
p(\alpha/k)=\frac{1}{1-\alpha/k}=\frac{k}{k-\alpha}.
\]
Equivalently, such a regularity estimate yields only a set estimate of the
form
\[
\mu(f\in A)
\le
C\bigl(\lambda_k(A)\bigr)^{\alpha/k}.
\]
Consequently, this route can give an estimate of the form \eqref{lorentz}
with exponent at most $\frac{1}{kd}$. This is weaker than the exponent
$\frac{1}{k(d-1)+1}$ obtained in Theorem~\ref{th-main-lorentz}. Thus, in the
vector-valued case, small-ball bounds, Lorentz-space estimates, and fractional
regularity estimates become genuinely distinct problems.

\subsection{The dimension-free fractional regularity estimate}

For arbitrary log-concave measures $\mu$ and vector-valued polynomial mappings
$f\in\mathcal P_d(\mathbb R^n;\mathbb R^k)$, we prove that, whenever
$\mu\circ f^{-1}$ is absolutely continuous, its density belongs to
$B_{1,\infty}^{\frac{1}{k(d-1)+1}}(\mathbb R^k)$.
To formulate this result explicitly, we use the following modulus of
continuity. For a finite Borel measure $\nu$ on $\mathbb R^k$, set
\[
\sigma(\nu,t):=
\sup\Bigl\{
\int_{\mathbb R^k}\partial_\theta\varphi\,d\nu\colon
|\theta|=1,\,
\varphi\in C_b^\infty(\mathbb R^k),\,
\|\varphi\|_\infty\le t,\,
\|\partial_\theta\varphi\|_\infty\le 1
\Bigr\}.
\]
In particular, we use an equivalent definition of the Nikolskii--Besov
seminorm based on this modulus of continuity, see \cite[Theorem~2.1]{Kos-MS} and \cite{Kos-FCAA}. Namely, for
$\varrho\in L^1(\mathbb R^k)$ and $\alpha\in(0,1)$,
\[
\|\varrho\|_{\dot B^\alpha_{1,\infty}(\mathbb R^k)}
:=
\sup_{t>0}
t^{-\alpha}\sigma(\varrho\,dx,t).
\]

We can now state our second main result in full detail.

\begin{theorem}\label{th-main-besov}
Let $0<a\le b<\infty$, $\eta>0$, $k,d\in\mathbb N$, and $d\ge2$. There exists
a constant $C:=C(a,b,\eta,k,d)>0$ such that, for every $n\in\mathbb N$, every
log-concave measure $\mu$ on $\mathbb R^n$, and every
$
f=(f_1,\ldots,f_k)\in\mathcal P_d(\mathbb R^n;\mathbb R^k)
$
satisfying
\[
a\le \operatorname{Var}_{\mu}(f_j)\le b
\quad
\forall j\in\{1,\ldots,k\}
\quad
\text{and}
\quad
\det \Sigma_\mu^d(f)\ge \eta,
\]
one has
\[
\sigma(\mu\circ f^{-1},t)
\le
C t^{\frac{1}{k(d-1)+1}}
\quad
\forall t>0.
\]
\end{theorem}

\begin{remark}
In Theorem~\ref{th-main-besov}, one may take the constant in the form
\[
C(a,b,\eta,k,d)
=
C'(\eta,k,d)
\bigl(a^{-k}b^{k-1}\bigr)^{\frac{1}{2(k(d-1)+1)}}.
\]
\end{remark}

\medskip

In particular, for $k=1$, we recover the dimension-free scalar
Nikolskii--Besov regularity estimate \eqref{eq-scalar-reg}, up to replacing
the factor $Cd$ by a $d$-dependent constant $C(d)$.

In the vector-valued setting, dimension-free regularity results of this type
were previously known only in the Gaussian case. Namely, for the standard
Gaussian measure $\gamma_n$ on $\mathbb R^n$, it was proved in \cite{KosZh}
that, for every $\alpha\in(0,\frac{1}{2k(d-1)})$,
\begin{equation}\label{Gauss-reg}
\sigma(\gamma_n\circ f^{-1},t)
\le
C(a,b,k,d,\alpha)\, t^\alpha
\quad \forall t>0,
\end{equation}
provided that
\[
\int_{\mathbb R^n}
\Delta_f
\, d\gamma_n\ge a
\quad\text{and}\quad
\max_{1\le j\le k}\operatorname{Var}_{\gamma_n}(f_j)\le b,
\]
where $\Delta_f:=\det(J_fJ_f^*)$.
See also \cite{BKZ} for earlier results in this setting. Thus, even in the
Gaussian case, Theorem~\ref{th-main-besov} gives a sharper regularity exponent than the previously known bound
\eqref{Gauss-reg}.

\subsection{Strategy of the proof}

To prove Theorems~\ref{th-main-lorentz} and~\ref{th-main-besov}, we first
choose a $k$-dimensional subspace $H\in Gr(k,n)$ and fix all variables in
$H^\perp$. We are then reduced to the simpler setting of a polynomial
mapping from $\mathbb R^k$ to $\mathbb R^k$, which allows us to obtain the
corresponding Lorentz-type and Besov-type estimates in terms of the
restricted Jacobian matrix involving only derivatives along $H$.

We next average these estimates over the Grassmannian $Gr(k,n)$. This yields
dimension-dependent analogues of Theorems~\ref{th-main-lorentz} and
\ref{th-main-besov}, with nondegeneracy measured by
$\det(J_fJ_f^*)$. For each fixed dimension $n$, we then show, by a
compactness argument, that the derivative-based nondegeneracy condition is
equivalent to the algebraic nondegeneracy condition expressed in terms of
the covariance matrix $\Sigma_\mu^d(f)$ of the extended mapping. This
transition is essential, as it allows us to apply the localization lemma of
Fradelizi and Gu\'edon~\cite{FrGue}, thereby removing the dependence on the
ambient dimension.

\subsection{Applications in probability}

Limit theorems and other properties of random vectors with polynomial
components have been extensively studied in the Gaussian setting, see, for
instance, \cite{PT04,NP05,NP09,NP12,HLN14,HMP}. In this setting, polynomial
components are precisely finite sums of elements of Wiener chaoses of bounded
order.
The main results of the present paper allow us to extend some of these results
to the setting of general log-concave measures, while also yielding new and
sharper estimates even in the Gaussian case.

For $d,k\in\mathbb N$, let $\mathcal P_{d,k}^{\rm lc}$ denote the class of all
$\mathbb R^k$-valued random vectors, considered up to equality in distribution,
that are limits in distribution of random vectors $X$ of the form
\[
X=f(\xi_1,\ldots,\xi_n),
\]
where $n\in\mathbb N$ is arbitrary,
$f\in\mathcal P_d(\mathbb R^n;\mathbb R^k)$, and
$\xi=(\xi_1,\ldots,\xi_n)$ is a random vector with a log-concave distribution.
Since finite-dimensional Gaussian distributions are log-concave, the class
$\mathcal P_{d,k}^{\rm lc}$ contains, in particular, all random vectors on an
abstract Wiener space whose components belong to finite sums of Wiener chaoses
of orders at most $d$. 

Let $\mathcal P_d^{\rm lc}:=\mathcal P_{d,1}^{\rm lc}$.
It can be verified that every random variable from $\mathcal P_d^{\rm lc}$
has finite moments of all orders. 
The variance $\operatorname{Var}$ and the covariance matrix $\Sigma^d$ of the extended random
vector are defined analogously to $\operatorname{Var}_\mu$ and
$\Sigma_\mu^d$, respectively, with integration with respect to $\mu$ replaced
by expectation.

Passing to the limit in Theorems~\ref{th-main-lorentz} and
\ref{th-main-besov} gives the following extension.

\begin{corollary}\label{cor-main}
Let $a,b\in(0,+\infty)$, $\eta>0$, $k,d\in\mathbb N$, and $d\ge2$. There exist
constants $C_1:=C_1(a,\eta,k,d)>0$ and $C_2:=C_2(a,b,\eta,k,d)>0$ such that,
for every random vector
$X=(X_1,\ldots,X_k)\in\mathcal P_{d,k}^{\rm lc}$
satisfying
\[
a\le \operatorname{Var}(X_j)\le b
\quad
\forall j\in\{1,\ldots,k\}
\quad
\text{and}
\quad
\det \Sigma^d(X)\ge \eta,
\]
one has
\[
\mathbb P(X\in A)
\le
C_1\bigl(\lambda_k(A)\bigr)^{\frac{1}{k(d-1)+1}}
\]
for every Borel set $A\subset\mathbb R^k$, and
\[
\sigma(\mathbb P\circ X^{-1},t)
\le
C_2t^{\frac{1}{k(d-1)+1}}
\quad
\forall t>0.
\]
\end{corollary}

In particular, Corollary~\ref{cor-main} extends the classical result of
Kusuoka \cite{Kusuoka83} beyond the Gaussian setting. Namely, for
$X=(X_1,\ldots,X_k)\in\mathcal P_{d,k}^{\rm lc}$, the distribution
$\mathbb P\circ X^{-1}$ is absolutely continuous if and only if there is no
nonzero polynomial $Q\in\mathcal P_{d^{k-1}}(\mathbb R^k)$ such that
\[
Q(X_1,\ldots,X_k)=0
\quad \text{almost surely}.
\]

Another implication of Corollary~\ref{cor-main} concerns the comparison
between weak and strong convergence for distributions of random vectors from
$\mathcal P_{d,k}^{\rm lc}$. For two $\mathbb R^k$-valued random vectors
$X$ and $Y$, we define the total variation distance between their
distributions by
\[
d_{\mathrm{TV}}(X,Y):=
\sup\Bigl\{
\bigl|\mathbb E[\varphi(X)]-\mathbb E[\varphi(Y)]\bigr| \colon
\varphi\in C_0^\infty(\mathbb R^k),\ \|\varphi\|_\infty\le 1
\Bigr\}
\]
and the bounded Kantorovich--Rubinstein, or Fortet--Mourier, distance by
\[
d_{\mathrm{KR}}(X,Y):=
\sup\Bigl\{
\bigl|\mathbb E[\varphi(X)]-\mathbb E[\varphi(Y)]\bigr| \colon
\varphi\in C_0^\infty(\mathbb R^k),\ \|\varphi\|_\infty\le 1,\ 
\|\nabla\varphi\|_\infty\le 1
\Bigr\}.
\]
The former distance coincides with the $L^1(\mathbb R^k)$ distance between
the densities whenever such densities exist. 
The latter distance metrizes convergence in distribution of the corresponding
random vectors. It is bounded from above by the usual Kantorovich distance,
also called the $1$-Wasserstein distance.

Nourdin and Poly \cite{NP} proved in the Gaussian case that if a
sequence $X_n$ from the finite sum of Wiener chaoses of orders at most $d$
converges in distribution to a nonconstant random variable $X_\infty$, then
the convergence holds in total variation and
\[
d_{\rm TV}(X_n,X_\infty)
\le
C d_{\rm KR}(X_n,X_\infty)^{\frac{1}{2d+1}}
\]
for some constant $C>0$.
The exponent was later improved in \cite{Kos18} to
$\frac{1}{d+1}$, and the result was extended from the Gaussian setting to the
setting of arbitrary log-concave measures.

In the vector-valued case, for a sequence $X_n$ of random vectors with
components from finite sums of Wiener chaoses of orders at most $d$ converging
in distribution to a random vector $X_\infty$ and satisfying the additional
assumption
\[
\mathbb E[\Delta_{X_n}]\ge a>0,
\]
where $\Delta_{X_n}$ is the determinant of the corresponding Malliavin matrix,
it was proved in \cite{NNP} that
\[
d_{\rm TV}(X_n,X_\infty)
\le
C d_{\rm KR}(X_n,X_\infty)^\alpha
\]
for every
\[
\alpha<
\alpha_0=
\frac{1}{(k+1)(4k(d-1)+3)+1}
\]
with some constant $C>0$. The exponent $\alpha_0$ was later improved in
\cite{BKZ} to
$\alpha_0=\frac{1}{4k(d-1)+1}$
and recently in \cite{KosZh} to
$\alpha_0=\frac{1}{2k(d-1)+1}$.
Corollary~\ref{cor-main} allows us to improve the exponent further, extend the
result to the log-concave setting, and drop the uniform Malliavin determinant
assumption, thereby obtaining the full analogue of the scalar Nourdin--Poly
result.

\begin{corollary}\label{cor-converg}
Let $a,b\in(0,+\infty)$, $\eta>0$, $k,d\in\mathbb N$, and $d\ge2$. There exists
a constant $C:=C(a,b,\eta,k,d)>0$ such that, for every
$X=(X_1,\ldots,X_k),Y=(Y_1,\ldots,Y_k)\in\mathcal P_{d,k}^{\rm lc}$
satisfying
\[
a\le \operatorname{Var}(X_j)\le b,
\quad
a\le \operatorname{Var}(Y_j)\le b
\quad
\forall j\in\{1,\ldots,k\},
\]
and
\[
\det \Sigma^d(X)\ge \eta,
\quad
\det \Sigma^d(Y)\ge \eta,
\]
one has
\[
d_{\rm TV}(X,Y)
\le
C d_{\rm KR}(X,Y)^{\frac{1}{k(d-1)+2}}.
\]

In particular, if a sequence
$X_n\in\mathcal P_{d,k}^{\rm lc}$ converges in distribution to an
$\mathbb R^k$-valued random vector $X_\infty$ whose distribution is absolutely
continuous, then the convergence holds in total variation and
\[
d_{\rm TV}(X_n,X_\infty)
\le
C d_{\rm KR}(X_n,X_\infty)^{\frac{1}{k(d-1)+2}}
\]
for some constant $C>0$ depending on the sequence and on the limiting
distribution.
\end{corollary}

\subsection{Notation}

Let $C^\infty(\mathbb R^n)$ denote the space of all smooth functions on
$\mathbb R^n$, let $C_0^\infty(\mathbb R^n)$ denote the space of all smooth
compactly supported functions on $\mathbb R^n$, and let
$C_b^\infty(\mathbb R^n)$ denote the space of all smooth functions on
$\mathbb R^n$ such that the function itself and all its partial derivatives
are bounded.

If $\mu$ is a Borel probability measure on
$\mathbb R^n$ and $f\colon \mathbb R^n\to \mathbb R^k$ is a Borel mapping, we
write $\mu\circ f^{-1}$ for the image measure, or pushforward, of $\mu$ under
$f$, that is,
\[
\mu\circ f^{-1}(A):=\mu\bigl(f^{-1}(A)\bigr)=\mu(f\in A)
\]
for every Borel set $A\subset\mathbb R^k$.
For $f\in L^2(\mu)$, let
\[
\operatorname{Var}_\mu(f)
:=
\int_{\mathbb R^n}
\Bigl(f-\int_{\mathbb R^n} f\,d\mu\Bigr)^2\,d\mu.
\]
For $f=(f_1,\ldots,f_k)$ with each $f_j\in L^2(\mu)$, let
\[
\langle \operatorname{Cov}_\mu(f)y,y\rangle
:=
\operatorname{Var}_\mu(\langle f,y\rangle)
\quad \forall y\in\mathbb R^k.
\]

We recall that a Borel probability measure $\mu$ on $\mathbb R^n$ is called
log-concave if
\[
\mu\bigl(s A+(1-s)B\bigr)
\ge
\bigl(\mu(A)\bigr)^s \bigl(\mu(B)\bigr)^{1-s}
\]
for all compact sets $A,B\subset\mathbb R^n$ and all $s\in[0,1]$.
If $\mu$ is absolutely continuous with respect to Lebesgue measure, then it is
log-concave if and only if its density is a log-concave function, that is, a
function of the form $e^{-V}$, where
$V\colon\mathbb R^n\to(-\infty,+\infty]$ is convex. In particular, the
normalized restriction of Lebesgue measure to any convex body in
$\mathbb R^n$ and Gaussian measures are log-concave.
We say that a probability measure on $\mathbb R^n$ is isotropic if it has zero
barycenter and identity covariance matrix.

For a mapping $f=(f_1,\ldots,f_k)$ with $f_j\in C^\infty(\mathbb R^n)$, let
\[
J_f=
\bigl(\partial_{x_j} f_i\bigr)_{\substack{1\le i\le k\\ 1\le j\le n}}
\]
be its Jacobian matrix, and let
\[
M_f = J_fJ_f^* =
\bigl(\langle \nabla f_i,\nabla f_j\rangle\bigr)_{1\le i,j\le k}
\]
be the Gram matrix of the gradients of the components of $f$. Let
\[
\Delta_f=\det M_f.
\]

Throughout the paper, constants are denoted by $C$ and may change from line to
line. Their dependence on parameters is always indicated explicitly, for
instance by writing $C(a,b,k,d)$. When several constants appear in the same
argument, we distinguish them by subscripts, writing $C_1,C_2$, etc. These
subscripts are only labels and do not indicate any additional dependence.

\subsection{Structure of the paper}

The rest of the paper is organized as follows. In
Section~\ref{sect-prelim}, we prove key technical lemmas establishing the
estimates of Theorems~\ref{th-main-lorentz} and~\ref{th-main-besov} with
dimension-dependent constants and with $\Delta_f$ as the nondegeneracy
parameter. In Section~\ref{sect-algeb}, for fixed $n\in\mathbb N$, we
establish a connection between $\Delta_f$ and $\Sigma_\mu^d(f)$ for isotropic
log-concave measures~$\mu$.
Section~\ref{sect-main} is devoted to the proof
of Theorems~\ref{th-main-lorentz} and~\ref{th-main-besov}, where the
dimension-dependent estimates are turned into dimension-free ones by means of
the localization lemma of Fradelizi and Gu\'edon~\cite{FrGue}.
Finally, in Section~\ref{sect-probab}, we prove the
corresponding probabilistic applications.

\section{Preliminary dimension-dependent estimates in terms of $\Delta_f$}\label{sect-prelim}

\subsection{Lemmas for Lorentz-norm estimates}

We first record a more convenient equivalent description of the
Lorentz-type estimate~\eqref{lorentz}.

\begin{lemma}\label{lem-lorentz-equivalence}
Let $\alpha\in(0,1]$ and let $\nu$ be a probability Borel measure on $\mathbb R^k$.
Then the estimate
\begin{equation}\label{eq-lorentz-1}
\nu(A)\le C\bigl(\lambda_k(A)\bigr)^\alpha
\end{equation}
for every Borel set $A\subset \mathbb R^k$ is equivalent to the estimate
\begin{equation}\label{eq-lorentz-2}
\int_{\mathbb R^k}\psi\,d\nu\le C\|\psi\|_{L^1(\mathbb R^k)}^\alpha
\end{equation}
for every $\psi\in C_0^\infty(\mathbb R^k)$ with $0\le\psi\le 1$.
\end{lemma}

\begin{proof}
Assume first that \eqref{eq-lorentz-1} holds. Then, by Fubini's theorem,
\[
\int_{\mathbb R^k}\psi\,d\nu
=
\int_0^1\nu(\psi\ge t)\,dt
\le
C\int_0^1\bigl(\lambda_k(\psi\ge t)\bigr)^\alpha\,dt
\le
C\biggl(\int_0^1\lambda_k(\psi\ge t)\,dt\biggr)^\alpha
=
C\|\psi\|_{L^1(\mathbb R^k)}^\alpha.
\]

Conversely, assume that \eqref{eq-lorentz-2} holds, and let
$A\subset\mathbb R^k$ be a Borel set. If $\lambda_k(A)=\infty$, then
\eqref{eq-lorentz-1} is trivial. Thus, we may assume that
$\lambda_k(A)<\infty$. Let $\varepsilon>0$. Since $\nu$ is a Radon measure,
there exists a compact set $K\subset A$ such that
\[
\nu(A)
\le
\nu(K)+\varepsilon.
\]
For $\delta>0$, let $u_\delta\in C_0^\infty(\mathbb R^k)$ be such that
\[
\operatorname{supp}u_\delta\subset B_\delta,
\qquad
u_\delta\ge0,
\qquad
\int_{\mathbb R^k}u_\delta(y)\,dy=1,
\]
where $B_\delta$ is the Euclidean ball of radius $\delta$ centered at the
origin. Put
\[
\psi_\delta:=I_{K+B_\delta}*u_\delta.
\]
Then $\psi_\delta\in C_0^\infty(\mathbb R^k)$ and $0\le\psi_\delta\le1$.
Moreover, $\psi_\delta=1$ on $K$. Indeed, if $y\in K$ and
$z\in\operatorname{supp}u_\delta$, then $y-z\in K+B_\delta$, and hence
$\psi_\delta(y)=1$.

Applying \eqref{eq-lorentz-2}, we obtain
\[
\nu(A)-\varepsilon
\le
\nu(K)
\le
\int_{\mathbb R^k}\psi_\delta\,d\nu
\le
C\|\psi_\delta\|_{L^1(\mathbb R^k)}^\alpha
\le
C\bigl(\lambda_k(K+B_{2\delta})\bigr)^\alpha.
\]
Letting $\delta\to0$, we get
\[
\nu(A)-\varepsilon
\le
C\bigl(\lambda_k(K)\bigr)^\alpha
\le
C\bigl(\lambda_k(A)\bigr)^\alpha.
\]
Finally, letting $\varepsilon\to0$, we obtain \eqref{eq-lorentz-1}.
\end{proof}

We now use the following lemma due to A.~Carbery, F.~Ricci, and J.~Wright.

\begin{lemma}[see {\cite[Lemma~3.2]{CRW03}} and {\cite[Lemma~4.8]{ChGGHIW25}}]
\label{lem-effective-inverse}
Let $k, d\in\mathbb{N}$, $f\in \mathcal{P}_d(\mathbb{R}^k; \mathbb{R}^k)$ such that $\det J_f$ is not identically zero.
Then there exist $\ell\le d^k$ open subsets
$U_1,\ldots,U_\ell$ of $\mathbb R^k$ such that

\begin{enumerate}
\item The sets $U_1,\ldots,U_\ell$ are pairwise disjoint and cover
$\mathbb R^k$ except for a set of Lebesgue measure zero.

\item For every $j\in\{1,\ldots,\ell\}$, the restriction $f|_{U_j}$ is a
diffeomorphism onto its image.
\end{enumerate}
\end{lemma}

\begin{corollary}\label{cor-lorentz}
Let $k,d\in\mathbb N$, and let $\varrho$ be a nonnegative integrable function
on $\mathbb R^k$ such that
\[
0\le \varrho(x)\le M
\quad \forall x\in\mathbb R^k.
\]
Let $f\in \mathcal P_d(\mathbb R^k;\mathbb R^k)$. Then, for every
$\varepsilon>0$ and every $\psi\in C_0^\infty(\mathbb R^k)$ with
$0\le \psi\le 1$, one has
\[
\int_{\mathbb R^k}\psi(f)\varrho\,dx
\le
\int_{\mathbb R^k}I_{\{|\det J_f|\le \varepsilon\}}\varrho\,dx
+
d^kM\varepsilon^{-1}\|\psi\|_{L^1(\mathbb R^k)}.
\]
\end{corollary}

\begin{proof}
If $\det J_f$ is identically zero, the estimate follows immediately from
$0\le\psi\le1$. Assume that $\det J_f$ is not identically zero. We have
\[
\int_{\mathbb R^k}\psi(f)\varrho\,dx
\le
\int_{\mathbb R^k}I_{\{|\det J_f|\le \varepsilon\}}\varrho\,dx
+
M\int_{\mathbb R^k}I_{\{|\det J_f|> \varepsilon\}}\psi(f)\,dx.
\]
For the second term,
\[
\int_{\mathbb R^k}I_{\{|\det J_f|> \varepsilon\}}\psi(f)\,dx
\le
\varepsilon^{-1}
\int_{\mathbb R^k}\psi(f)|\det J_f|\,dx.
\]
By Lemma~\ref{lem-effective-inverse}, there exist pairwise disjoint open sets
$U_1,\ldots,U_\ell$, with $\ell\le d^k$, which cover $\mathbb R^k$ up to a
set of Lebesgue measure zero, and such that $f|_{U_j}$ is a diffeomorphism
onto its image for every $j\in\{1,\ldots,\ell\}$. Therefore, by the change of
variables formula,
\[
\int_{\mathbb R^k}\psi(f)|\det J_f|\,dx
=
\sum_{j=1}^{\ell}
\int_{U_j}\psi(f)|\det J_f|\,dx
=
\sum_{j=1}^{\ell}
\int_{f(U_j)}\psi(y)\,dy
\le
d^k\int_{\mathbb R^k}\psi(y)\,dy.
\]
Combining the preceding estimates gives the desired bound.
\end{proof}

We will also use the following elementary observation concerning averages over the Grassmannian.

\begin{lemma}\label{lem-Gr-average}
Let $k, n\in \mathbb{N}$, $n\ge k$, let $\nu_{k,n}$ be the normalized Haar measure on the Grassmannian
$Gr(k,n)$ of all $k$-dimensional subspaces of $\mathbb R^n$. Then there
exists a constant $C(k,n)>0$ such that, for every $k\times n$ matrix $J$,
one has
\[
\int_{Gr(k,n)}(\det J^H)^2\,\nu_{k,n}(dH)
=
C(k,n)\det(JJ^*),
\]
where $J^H:=JU_H$ and $U_H$ is the $n\times k$ matrix whose columns form an
orthonormal basis in $H$.
\end{lemma}

\begin{proof}
Let $H\in Gr(k,n)$, and let $e_1,\ldots,e_k$ be an orthonormal basis in $H$.
Denote by $U_H$ the $n\times k$ matrix whose columns are $e_1,\ldots,e_k$.
The quantity
\[
\bigl(\det(JU_H)\bigr)^2
\]
does not depend on the particular orthonormal basis chosen in $H$. Indeed, if
another orthonormal basis is chosen, then $U_H$ is replaced by $U_HO$ for
some orthogonal $k\times k$ matrix $O$, and therefore
\[
\bigl(\det(JU_HO)\bigr)^2
=
\bigl(\det(JU_H)\bigr)^2(\det O)^2
=
\bigl(\det(JU_H)\bigr)^2.
\]

We have
\[
(\det J^H)^2
=
(\det(JU_H))^2
=
\det(JU_H)\det(U_H^*J^*)
=
\det(JU_HU_H^*J^*).
\]
Let
\[
J=S\Sigma V^*
\]
be a singular value decomposition of $J$, where $S$ and $V$ are orthogonal
matrices and
\[
\Sigma
=
\begin{pmatrix}
\sigma_1 & 0 & \ldots & 0 & 0 & \ldots & 0\\
0 & \sigma_2 & \ldots & 0 & 0 & \ldots & 0\\
\ldots & \ldots & \ldots & \ldots & \ldots & \ldots & \ldots\\
0 & 0 & \ldots & \sigma_k & 0 & \ldots & 0
\end{pmatrix}.
\]
Since $J^*=V\Sigma^*S^*$, we have
\[
JU_HU_H^*J^*
=
S\Sigma V^*U_HU_H^*V\Sigma^*S^*.
\]
The matrix $V^*U_H$ has orthonormal columns and corresponds to the subspace
$V^*H$. Hence
\[
V^*U_HU_H^*V
=
U_{V^*H}U_{V^*H}^*.
\]
Thus,
\[
JU_HU_H^*J^*
=
S\Sigma U_{V^*H}U_{V^*H}^*\Sigma^*S^*.
\]
Taking determinants and using the fact that $S$ is orthogonal, we obtain
\[
\det(JU_HU_H^*J^*)
=
\det(\Sigma U_{V^*H}U_{V^*H}^*\Sigma^*).
\]

Let
\[
\Sigma_0
=
\begin{pmatrix}
1 & 0 & \ldots & 0 & 0 & \ldots & 0\\
0 & 1 & \ldots & 0 & 0 & \ldots & 0\\
\ldots & \ldots & \ldots & \ldots & \ldots & \ldots & \ldots\\
0 & 0 & \ldots & 1 & 0 & \ldots & 0
\end{pmatrix}.
\]
Then
\[
\Sigma=D\Sigma_0,
\quad
D=\operatorname{diag}(\sigma_1,\ldots,\sigma_k).
\]
Therefore,
\[
\Sigma U_{V^*H}U_{V^*H}^*\Sigma^*
=
D\Sigma_0 U_{V^*H}U_{V^*H}^*\Sigma_0^*D.
\]
Consequently,
\[
\det(\Sigma U_{V^*H}U_{V^*H}^*\Sigma^*)
=
\sigma_1^2\ldots\sigma_k^2
\det(\Sigma_0 U_{V^*H}U_{V^*H}^*\Sigma_0^*).
\]
Since
\[
\sigma_1^2\ldots\sigma_k^2=\det(JJ^*),
\]
we get
\[
(\det J^H)^2
=
\det(JJ^*)
\det(\Sigma_0 U_{V^*H}U_{V^*H}^*\Sigma_0^*).
\]
Using the invariance of $\nu_{k,n}$ under orthogonal transformations, we
obtain
\begin{align*}
\int_{Gr(k,n)}(\det J^H)^2\,\nu_{k,n}(dH)
&=
\det(JJ^*)
\int_{Gr(k,n)}
\det(\Sigma_0 U_{V^*H}U_{V^*H}^*\Sigma_0^*)\,\nu_{k,n}(dH)
\\
&=
\det(JJ^*)
\int_{Gr(k,n)}
\det(\Sigma_0 U_HU_H^*\Sigma_0^*)\,\nu_{k,n}(dH).
\end{align*}
The last integral depends only on $k$ and $n$. Denote it by $C(k,n)$. It is
positive, because the integrand is continuous, nonnegative, and strictly
positive in a neighborhood of
\[
H=\operatorname{span}\{e_1,\ldots,e_k\}.
\]
Therefore,
\[
\int_{Gr(k,n)}(\det J^H)^2\,\nu_{k,n}(dH)
=
C(k,n)\det(JJ^*),
\]
which completes the proof.
\end{proof}

We now apply the preceding two statements in the isotropic log-concave setting.

\begin{lemma}\label{lem-lorentz-dimensional}
Let $n,k,d\in\mathbb N$, $n\ge k$, and $d\ge2$. There exists a constant
$C(k,d,n)>0$, depending only on $k,d,n$, such that, for every isotropic
log-concave measure $\mu$ on $\mathbb R^n$, every
$f=(f_1,\ldots,f_k)\in\mathcal P_d(\mathbb R^n;\mathbb R^k)$, and every
$\psi\in C_0^\infty(\mathbb R^k)$ with $0\le \psi\le1$, one has
\[
\biggl(\int_{\mathbb R^n}\Delta_f\,d\mu\biggr)^{\frac{1}{2(k(d-1)+1)}}
\biggl(\int_{\mathbb R^n}\psi(f)\,d\mu\biggr)
\le
C(k,d,n)
\|\psi\|_{L^1(\mathbb R^k)}^{\frac{1}{k(d-1)+1}}.
\]
\end{lemma}

\begin{proof}
Let $\varrho$ be the density of $\mu$. Combining
\cite[Corollary~2.4]{Klartag07} with
\cite[Theorem~5.14(d)]{LV07}, we obtain constants
$c_1(n),c_2(n)>0$, depending only on $n$, such that
\begin{equation}\label{eq-exp-bound}
\varrho(x)\le c_1(n)e^{-c_2(n)|x|}
\quad \forall x\in\mathbb R^n.
\end{equation}
In particular, all sections considered below are bounded and integrable.

Fix a $k$-dimensional subspace $H\subset\mathbb R^n$ and an orthonormal basis
$e_1,\ldots,e_k$ in $H$. Write
\[
x=y+z,
\quad
y\in H,
\quad
z\in H^\perp.
\]
For every fixed $z\in H^\perp$, we apply Corollary~\ref{cor-lorentz} on the
subspace $H$ to the mapping $y\mapsto f(y+z)$ and to the density
$y\mapsto \varrho(y+z)$. For every $\varepsilon>0$, this gives
\[
\int_H \psi(f(y+z))\varrho(y+z)\,dy
\le
\int_H I_{\{|\det J_f^H(y+z)|\le\varepsilon\}}\varrho(y+z)\,dy
+
d^k\varepsilon^{-1}
\|\psi\|_{L^1(\mathbb R^k)}
\sup_{y\in H}\varrho(y+z).
\]
Integrating with respect to $z\in H^\perp$, we obtain
\[
\int_{\mathbb R^n}\psi(f)\,d\mu
\le
\mu\bigl(|\det J_f^H|\le\varepsilon\bigr)
+
d^k\varepsilon^{-1}
\|\psi\|_{L^1(\mathbb R^k)}
\int_{H^\perp}\sup_{y\in H}\varrho(y+z)\,dz.
\]
By \eqref{eq-exp-bound},
\[
\int_{H^\perp}\sup_{y\in H}\varrho(y+z)\,dz
\le
c_1(n)\int_{H^\perp}e^{-c_2(n)|z|}\,dz
\le C_1(k,n).
\]
Thus,
\begin{equation}\label{eq-lorentz-dimensional-H}
\int_{\mathbb R^n}\psi(f)\,d\mu
\le
\mu\bigl(|\det J_f^H|\le\varepsilon\bigr)
+
C_2(k,d,n)\varepsilon^{-1}
\|\psi\|_{L^1(\mathbb R^k)}.
\end{equation}

Assume first that
\[
\int_{\mathbb R^n}(\det J_f^H)^2\,d\mu>0.
\]
Since
\[
\det J_f^H\in\mathcal P_{k(d-1)}(\mathbb R^n),
\]
the Carbery--Wright inequality \eqref{CW-est} gives
\[
\mu\bigl(|\det J_f^H|\le\varepsilon\bigr)
\le
C_3(k,d)
\varepsilon^{\frac{1}{k(d-1)}}
\biggl(\int_{\mathbb R^n}(\det J_f^H)^2\,d\mu\biggr)^{-\frac{1}{2k(d-1)}}.
\]
Combining this with \eqref{eq-lorentz-dimensional-H} and choosing
\[
\varepsilon
=
\biggl(
\|\psi\|_{L^1(\mathbb R^k)}
\biggl(\int_{\mathbb R^n}(\det J_f^H)^2\,d\mu\biggr)^{\frac{1}{2k(d-1)}}
\biggr)^{\frac{k(d-1)}{k(d-1)+1}},
\]
we obtain
\[
\int_{\mathbb R^n}\psi(f)\,d\mu
\le
C_4(k,d,n)
\|\psi\|_{L^1(\mathbb R^k)}^{\frac{1}{k(d-1)+1}}
\biggl(\int_{\mathbb R^n}(\det J_f^H)^2\,d\mu\biggr)^{
-\frac{1}{2(k(d-1)+1)}
}.
\]
Equivalently,
\[
\biggl(\int_{\mathbb R^n}(\det J_f^H)^2\,d\mu\biggr)
\biggl(
\int_{\mathbb R^n}\psi(f)\,d\mu
\biggr)^{2(k(d-1)+1)}
\le
C_5(k,d,n)
\|\psi\|_{L^1(\mathbb R^k)}^2.
\]
This estimate is also valid when the first integral on the left-hand side is zero.

Averaging the last estimate over all $H\in Gr(k,n)$ with respect to the
normalized Haar measure $\nu_{k,n}$ and applying Lemma~\ref{lem-Gr-average},
we obtain
\[
\int_{Gr(k,n)}
\int_{\mathbb R^n}(\det J_f^H)^2\,d\mu\,\nu_{k,n}(dH)
=
C_6(k,n)\int_{\mathbb R^n}\Delta_f\,d\mu
\]
and hence
\[
\biggl(\int_{\mathbb R^n}\Delta_f\,d\mu\biggr)
\biggl(
\int_{\mathbb R^n}\psi(f)\,d\mu
\biggr)^{2(k(d-1)+1)}
\le
C_7(k,d,n)
\|\psi\|_{L^1(\mathbb R^k)}^2.
\]
Taking the power $\frac{1}{2(k(d-1)+1)}$ gives the claimed bound.
\end{proof}

\subsection{Lemmas for regularity estimates}

\begin{lemma}\label{key-lemma-1}
Let $k,d\in\mathbb N$, and let $\varrho\in C_0^\infty(\mathbb R^k)$ be a
nonnegative function. Let
$f\in \mathcal P_d(\mathbb R^k;\mathbb R^k)$,
and let $A_f:=\operatorname{adj}(J_f)$ denote the adjugate matrix of the Jacobian $J_f$.
Then, for every $\varepsilon>0$, every unit vector $\theta\in\mathbb R^k$,
every $t>0$, and every $\varphi\in C_b^\infty(\mathbb R^k)$ satisfying
\[
\|\varphi\|_\infty\le t,
\quad
\|\partial_\theta\varphi\|_\infty\le 1,
\]
one has
\begin{align*}
\biggl|
\int_{\mathbb R^k}(\partial_\theta\varphi)(f)\varrho\,dx
\biggr|
&\le
\int_{\mathbb R^k} I_{\{(\det J_f)^2\le 2\varepsilon^2\}}\varrho\,dx
\\
&\quad
+
t\varepsilon^{-1}
\int_{\mathbb R^k}
\bigl|\langle \nabla\varrho,A_f\theta\rangle\bigr|\,dx
+
18t
\int_{\mathbb R^k}
\frac{\bigl|\langle \nabla\det J_f,A_f\theta\rangle\bigr|}
{(\det J_f)^2+\varepsilon^2}
\varrho\,dx.
\end{align*}
\end{lemma}

\begin{proof}
Fix a function $\eta\in C^\infty(\mathbb R)$ such that
\begin{equation}\label{eq-Phi}
\Phi(s)=0 \quad \text{for } s\in[-1,1], \quad
\Phi(s)=1 \quad \text{for } s\notin[-2,2], \quad
0\le \Phi(s)\le 1 \quad \text{for all } s\in\mathbb R.
\end{equation}
It can be verified that one can choose such a function $\eta$ so that
$\|\Phi'\|_\infty\le 2$.
For $\varepsilon>0$, define
\[
\Phi_\varepsilon(s):=\Phi(s/\varepsilon^2),
\quad \forall s\in\mathbb R.
\]

We split the integral as follows:
\[
\int_{\mathbb R^k}(\partial_\theta\varphi)(f)\varrho\,dx
=
\int_{\mathbb R^k}
\bigl(1-\Phi_\varepsilon\bigl((\det J_f)^2\bigr)\bigr)
(\partial_\theta\varphi)(f)\varrho\,dx
+
\int_{\mathbb R^k}
\Phi_\varepsilon\bigl((\det J_f)^2\bigr)
(\partial_\theta\varphi)(f)\varrho\,dx.
\]
For the first term, using $\|\partial_\theta\varphi\|_\infty\le 1$, we have
\[
\biggl|
\int_{\mathbb R^k}
\bigl(1-\Phi_\varepsilon\bigl((\det J_f)^2\bigr)\bigr)
(\partial_\theta\varphi)(f)\varrho\,dx
\biggr|
\le
\int_{\mathbb R^k}I_{\{(\det J_f)^2\le 2\varepsilon^2\}}\varrho\,dx.
\]

We now consider the second term. For $x\in\mathbb R^k$ such that
$\det J_f(x)\ne0$, using
\[
A_f=(\det J_f)J_f^{-1},
\]
we obtain
\[
(\partial_\theta\varphi)(f)\det J_f
=
\langle \nabla(\varphi\circ f),A_f\theta\rangle.
\]
By the Piola identity applied to the cofactor matrix, see, for instance,
\cite[Section~8.1.4(b)]{EvansPDE} or \cite{KSh19}, the columns of $A_f$ are divergence-free:
\begin{equation}\label{eq-Piola}
\sum_{i=1}^k \partial_{x_i}(A_f)_{ij}=0
\quad \forall j\in\{1,\ldots,k\}.
\end{equation}

Set
\[
\Psi_\varepsilon(u):=
\begin{cases}
\dfrac{\Phi_\varepsilon(u^2)}{u}, & u\ne 0,\\[1ex]
0, & u=0.
\end{cases}
\]
The function $\Psi_\varepsilon$ is smooth and bounded, because
$\Phi_\varepsilon(u^2)=0$ for $|u|\le \varepsilon$.
Thus,
\begin{align*}
\int_{\mathbb R^k}
\Phi_\varepsilon\bigl((\det J_f)^2\bigr)(\partial_\theta\varphi)(f)\varrho\,dx
&=
\int_{\mathbb R^k}
\Psi_\varepsilon(\det J_f)
\bigl((\partial_\theta\varphi)(f)\det J_f\bigr)\varrho\,dx
\\
&=
\int_{\mathbb R^k}
\Psi_\varepsilon(\det J_f)
\langle \nabla(\varphi\circ f),A_f\theta\rangle\varrho\,dx.
\end{align*}
Since $\varrho$ is compactly supported, integration by parts gives
\begin{align*}
\int_{\mathbb R^k}
\Psi_\varepsilon(\det J_f)
\langle \nabla(\varphi\circ f),A_f\theta\rangle\varrho\,dx
&=
-\int_{\mathbb R^k}
\varphi(f)\,
\operatorname{div}\bigl(\Psi_\varepsilon(\det J_f)\varrho A_f\theta\bigr)\,dx.
\end{align*}
Expanding the divergence and using \eqref{eq-Piola}, we obtain
\begin{align*}
\int_{\mathbb R^k}
\Phi_\varepsilon\bigl((\det J_f)^2\bigr)(\partial_\theta\varphi)(f)\varrho\,dx
&=
-\int_{\mathbb R^k}
\varphi(f)\Psi_\varepsilon'(\det J_f)
\langle \nabla\det J_f,A_f\theta\rangle\varrho\,dx
\\
&\quad
-\int_{\mathbb R^k}
\varphi(f)\Psi_\varepsilon(\det J_f)
\langle \nabla\varrho,A_f\theta\rangle\,dx.
\end{align*}	
Now, by the definitions of $\Psi_\varepsilon$ and $\Phi_\varepsilon$, we have
\[
|\Psi_\varepsilon(u)|
\le \varepsilon^{-1},
\quad \forall u\in\mathbb R.
\]
Indeed, $\Psi_\varepsilon(u)=0$ for $|u|\le\varepsilon$, while
$|\Psi_\varepsilon(u)|\le |u|^{-1}$ for $|u|>\varepsilon$. Hence,
\[
\biggl|
\int_{\mathbb R^k}
\varphi(f)\Psi_\varepsilon(\det J_f)
\langle \nabla\varrho,A_f\theta\rangle\,dx
\biggr|
\le
t\varepsilon^{-1}
\int_{\mathbb R^k}
\bigl|\langle \nabla\varrho,A_f\theta\rangle\bigr|\,dx.
\]
Next, for $u\ne0$,
\[
\Psi_\varepsilon'(u)
=
2\Phi_\varepsilon'(u^2)
-
\frac{\Phi_\varepsilon(u^2)}{u^2}.
\]
Since
\[
\Phi_\varepsilon'(u^2)
=
\varepsilon^{-2}\Phi'(u^2/\varepsilon^2),
\]
and since $\Phi'$ is supported in $\{1\le |s|\le 2\}$, we obtain
\[
\bigl|\Psi_\varepsilon'(u)\bigr|
\le
4\varepsilon^{-2}I_{\{2\varepsilon^2\ge u^2\ge \varepsilon^2\}}
+
u^{-2}I_{\{u^2\ge \varepsilon^2\}}.
\]
Using $\varepsilon^{-2}\le 2u^{-2}$ on the set
$\{2\varepsilon^2\ge u^2\ge \varepsilon^2\}$, we obtain
\[
\bigl|\Psi_\varepsilon'(u)\bigr|
\le
9u^{-2}I_{\{u^2\ge \varepsilon^2\}}.
\]
Finally, on $\{u^2\ge\varepsilon^2\}$ one has
\[
u^{-2}\le \frac{2}{u^2+\varepsilon^2}.
\]
Therefore,
\[
\bigl|\Psi_\varepsilon'(u)\bigr|
\le
\frac{18}{u^2+\varepsilon^2},
\quad \forall u\in\mathbb R.
\]
Consequently,
\[
\biggl|
\int_{\mathbb R^k}
\varphi(f)\Psi_\varepsilon'(\det J_f)
\langle \nabla\det J_f,A_f\theta\rangle\varrho\,dx
\biggr|
\le
18t
\int_{\mathbb R^k}
\frac{\bigl|\langle \nabla\det J_f,A_f\theta\rangle\bigr|}
{(\det J_f)^2+\varepsilon^2}
\varrho\,dx.
\]
Combining the preceding estimates proves the claimed bound.
\end{proof}

\begin{definition}
We say that a nonnegative function $\varrho$ on $\mathbb R$ is unimodal if
there exists $a\in\mathbb R$ such that $\varrho$ is nondecreasing on
$(-\infty,a]$ and nonincreasing on $[a,\infty)$.	
\end{definition}	

For our purposes, it is important that every integrable
log-concave function on $\mathbb R$ is unimodal.

We will use the following integration-by-parts estimate for unimodal integrable functions.

\begin{lemma}\label{lem-deriv-est}
Let $\varrho\in C^\infty(\mathbb R)$ be a nonnegative, integrable,
unimodal function on $\mathbb R$. Then, for every $d\in\mathbb N$ and every
$f\in \mathcal P_d(\mathbb R)\cap L^1(\varrho\,ds)$, one has
\[
\int_{\mathbb R}|f(s)|\,|\varrho'(s)|\,ds
\le
2\sup_{s\in\mathbb R}|f(s)\varrho(s)|
+
\int_{\mathbb R}|f'(s)|\varrho(s)\,ds.
\]
\end{lemma}

\begin{proof}
If $f\equiv 0$, there is nothing to prove. 
Let $\eta\in C_0^\infty(\mathbb R)$
be a function such that
\begin{equation}\label{eq-eta}
\eta(s)=1 \quad \text{for } s\in[-1,1], \quad
\eta(s)=0 \quad \text{for } s\notin[-2,2], \quad
0\le \eta(s)\le 1 \quad \text{for all } s\in\mathbb R,
\end{equation}
and let
$\eta_r(s):=\eta(s/r)$.

Let $a\in\mathbb R$ be such that $\varrho$ is nondecreasing on
$(-\infty,a]$ and nonincreasing on $[a,\infty)$. Hence
\[
\varrho'(s)\ge 0 \quad \text{for } s<a
\quad\text{and}\quad
\varrho'(s)\le 0 \quad \text{for } s>a.
\]
Let $b_1,\ldots,b_m$ be the distinct real zeros of $f$. On each connected
component of
\[
(-\infty,a)\setminus\{b_1,\ldots,b_m\}
\]
the polynomial $f$ has a constant sign. Denote these components by
$I_\ell^-$. For each such interval, choose $\sigma_\ell^-\in\{-1,1\}$ so
that
\[
|f(s)|=\sigma_\ell^- f(s)
\quad \forall s\in I_\ell^-.
\]
Since $\varrho'\ge0$ on $(-\infty,a)$, we have
\[
\int_{-\infty}^a |f(s)|\,|\varrho'(s)|\eta_r(s)\,ds
=
\sum_\ell
\int_{I_\ell^-}\sigma_\ell^- f(s)\varrho'(s)\eta_r(s)\,ds.
\]
Integrating by parts on each interval $I_\ell^-$ and summing over $\ell$,
all boundary terms at zeros of $f$ vanish. The boundary term at $-\infty$
vanishes because $\eta_r$ is compactly supported. Thus the only possible
remaining boundary term is at $a$, and therefore
\[
\int_{-\infty}^a |f(s)|\,|\varrho'(s)|\eta_r(s)\,ds
\le
|f(a)|\varrho(a)\eta_r(a)
+
\int_{-\infty}^a |f'(s)|\varrho(s)\eta_r(s)\,ds
+
\int_{-\infty}^a |f(s)|\varrho(s)|\eta_r'(s)|\,ds.
\]
Similarly, on each connected component of
\[
(a,\infty)\setminus\{b_1,\ldots,b_m\}
\]
the polynomial $f$ has a constant sign. Since $\varrho'\le0$ on
$(a,\infty)$, the same integration by parts gives
\[
\int_a^\infty |f(s)|\,|\varrho'(s)|\eta_r(s)\,ds
\le
|f(a)|\varrho(a)\eta_r(a)
+
\int_a^\infty |f'(s)|\varrho(s)\eta_r(s)\,ds
+
\int_a^\infty |f(s)|\varrho(s)|\eta_r'(s)|\,ds.
\]
Adding the two estimates and using
\[
|\eta_r'(s)|\le r^{-1}\|\eta'\|_\infty
\]
we obtain
\[
\int_{\mathbb R}|f(s)|\,|\varrho'(s)|\eta_r(s)\,ds
\le
2\sup_{s\in\mathbb R}|f(s)|\varrho(s)
+
\int_{\mathbb R}|f'(s)|\varrho(s)\,ds
+
r^{-1}\|\eta'\|_\infty
\int_{\mathbb R}|f(s)|\varrho(s)\,ds.
\]
Letting $r\to\infty$ and applying Fatou's lemma proves the claimed estimate.
\end{proof}

\begin{lemma}\label{key-lemma-2}
Let $k,d\in\mathbb N$, let $\theta\in\mathbb R^k$ satisfy $|\theta|=1$, and
let $\varrho\in C^\infty(\mathbb R^k)$ be a nonnegative integrable
log-concave function on $\mathbb R^k$. Let
$
f\in\mathcal P_d(\mathbb R^k;\mathbb R^k),
$
let $J_f$ be the Jacobian matrix of $f$, let
$
A_f:=\operatorname{adj}(J_f),
$
and set
\[
Q_j(x):=\langle A_f(x)\theta,e_j\rangle,
\quad \forall j\in\{1,\ldots,k\}.
\]
Then, for every $\varepsilon>0$, every $t>0$, and every
$\varphi\in C_b^\infty(\mathbb R^k)$ satisfying
\[
\|\varphi\|_\infty\le t,
\quad
\|\partial_\theta\varphi\|_\infty\le 1,
\]
one has
\begin{align*}
\biggl|
\int_{\mathbb R^k}(\partial_\theta\varphi)(f)\varrho\,dx
\biggr|
&\le
\int_{\mathbb R^k}
I_{\{(\det J_f)^2\le 2\varepsilon^2\}}\varrho\,dx
+
t\varepsilon^{-1}
\sum_{j=1}^k
\int_{\mathbb R^k}|\partial_{x_j}Q_j|\varrho\,dx
\\
&\quad
+
(2+18\pi kd)t\varepsilon^{-1}
\sum_{j=1}^k
\int_{\mathbb R^{k-1}}
\sup_{x_j\in\mathbb R}|Q_j(x)\varrho(x)|\,d\hat{x}^j,
\end{align*}
where
\[
\hat{x}^j:=(x_1,\ldots,x_{j-1},x_{j+1},\ldots,x_k)
\in\mathbb R^{k-1},
\quad j=1,\ldots,k.
\]
\end{lemma}

\begin{proof}
Let $\eta\in C_0^\infty(\mathbb R)$ be the function given in
\eqref{eq-eta}, and set
\[
\eta_r(x):=\eta(|x|/r)\in C^\infty_0(\mathbb{R}^k).
\]
By Lemma~\ref{key-lemma-1}, we obtain
\begin{align}\label{eq-lem-est}
\biggl|
\int_{\mathbb R^k}(\partial_\theta\varphi)(f)\varrho\eta_r\,dx
\biggr|
&\le
\int_{\mathbb R^k}
I_{\{(\det J_f)^2\le 2\varepsilon^2\}}\varrho\eta_r\,dx
\\
&\quad
+
t\varepsilon^{-1}
\int_{\mathbb R^k}
\bigl|\langle\nabla(\varrho\eta_r),A_f\theta\rangle\bigr|\,dx
+
18t
\int_{\mathbb R^k}
\frac{\bigl|\langle\nabla\det J_f,A_f\theta\rangle\bigr|}
{(\det J_f)^2+\varepsilon^2}
\varrho\eta_r\,dx.
\notag
\end{align}

We first estimate the term containing $\nabla(\varrho\eta_r)$. Since
\[
\nabla(\varrho\eta_r)=\eta_r\nabla\varrho+\varrho\nabla\eta_r
\]
and
\[
|\nabla\eta_r(x)|\le r^{-1}\|\eta'\|_\infty,
\]
we have
\begin{equation}\label{eq-grad-rho-eta}
\int_{\mathbb R^k}
\bigl|\langle\nabla(\varrho\eta_r),A_f\theta\rangle\bigr|\,dx
\le
\sum_{j=1}^k
\int_{\mathbb R^k}
|\partial_{x_j}\varrho|\,|Q_j|\eta_r\,dx
+
r^{-1}\|\eta'\|_\infty
\int_{\mathbb R^k}|A_f\theta|\varrho\,dx.
\end{equation}
For each $j$, by Fubini's theorem and by Lemma~\ref{lem-deriv-est} applied,
for fixed $\hat x^j$, to the one-dimensional unimodal integrable function
$x_j\mapsto\varrho(x)$, we obtain
\begin{align*}
\int_{\mathbb R^k}
|\partial_{x_j}\varrho|\,|Q_j|\eta_r\,dx
&\le
\int_{\mathbb R^k}
|\partial_{x_j}\varrho|\,|Q_j|\,dx
=
\int_{\mathbb R^{k-1}}
\int_{\mathbb R}
|\partial_{x_j}\varrho|\,|Q_j|\,dx_j\,d\hat{x}^j
\\
&\le
2\int_{\mathbb R^{k-1}}
\sup_{x_j\in\mathbb R}|Q_j(x)\varrho(x)|\,d\hat{x}^j
+
\int_{\mathbb R^k}|\partial_{x_j}Q_j|\varrho\,dx.
\end{align*}
Substituting this into \eqref{eq-grad-rho-eta}, we get
\begin{align}\label{eq-grad-term-est}
\int_{\mathbb R^k}
\bigl|\langle\nabla(\varrho\eta_r),A_f\theta\rangle\bigr|\,dx
&\le
2\sum_{j=1}^k
\int_{\mathbb R^{k-1}}
\sup_{x_j\in\mathbb R}|Q_j(x)\varrho(x)|\,d\hat{x}^j
\\
&
+
\sum_{j=1}^k
\int_{\mathbb R^k}|\partial_{x_j}Q_j|\varrho\,dx
+
r^{-1}\|\eta'\|_\infty
\int_{\mathbb R^k}|A_f\theta|\varrho\,dx.
\notag
\end{align}

We now estimate the term containing $\nabla\det J_f$. Since
\[
\langle\nabla\det J_f,A_f\theta\rangle
=
\sum_{j=1}^k Q_j\, \partial_{x_j}\det J_f,
\]
and $0\le \eta_r\le 1$, it is sufficient to estimate, for each fixed $j$,
\[
\int_{\mathbb R^k}
\frac{|\partial_{x_j}\det J_f|}
{(\det J_f)^2+\varepsilon^2}
|Q_j|\varrho\,dx.
\]
By Fubini's theorem,
\begin{align*}
\int_{\mathbb R^k}
\frac{|\partial_{x_j}\det J_f|}
{(\det J_f)^2+\varepsilon^2}
|Q_j|\varrho\,dx
&=
\int_{\mathbb R^{k-1}}
\int_{\mathbb R}
\frac{|\partial_{x_j}\det J_f|}
{(\det J_f)^2+\varepsilon^2}
|Q_j|\varrho\,dx_j\,d\hat{x}^j
\\
&
\le
\int_{\mathbb R^{k-1}}
\biggl(
\int_{\mathbb R}
\frac{|\partial_{x_j}\det J_f|}
{(\det J_f)^2+\varepsilon^2}\,dx_j
\biggr)
\sup_{x_j\in\mathbb R}|Q_j(x)\varrho(x)|\,d\hat{x}^j.
\end{align*}
Fix $\hat{x}^j$. The function $x_j\mapsto\det J_f(x)$ is a polynomial of
degree at most $k(d-1)\le kd$. Hence $x_j\mapsto\partial_{x_j}\det J_f(x)$
has at most $kd$ intervals of constant sign. On each such interval $I_\ell$,
\[
\int_{I_\ell}
\frac{|\partial_{x_j}\det J_f|}
{(\det J_f)^2+\varepsilon^2}\,dx_j
=
\biggl|
\int_{I_\ell}
\frac{\partial_{x_j}\det J_f}
{(\det J_f)^2+\varepsilon^2}\,dx_j
\biggr|
=
\varepsilon^{-1}
\biggl|
\int_{I_\ell}
\partial_{x_j}\arctan(\varepsilon^{-1}\det J_f)\,dx_j
\biggr|
\le
\pi\varepsilon^{-1}.
\]
Therefore,
\[
\int_{\mathbb R}
\frac{|\partial_{x_j}\det J_f|}
{(\det J_f)^2+\varepsilon^2}\,dx_j
\le
\pi kd\,\varepsilon^{-1}.
\]
Consequently,
\[
\int_{\mathbb R^k}
\frac{|\partial_{x_j}\det J_f|}
{(\det J_f)^2+\varepsilon^2}
|Q_j|\varrho\,dx
\le
\pi kd\,\varepsilon^{-1}
\int_{\mathbb R^{k-1}}
\sup_{x_j\in\mathbb R}|Q_j(x)\varrho(x)|\,d\hat{x}^j.
\]
Summing over $j$ yields
\begin{equation}\label{eq-det-term-est}
\int_{\mathbb R^k}
\frac{\bigl|\langle\nabla\det J_f,A_f\theta\rangle\bigr|}
{(\det J_f)^2+\varepsilon^2}
\varrho\eta_r\,dx
\le
\pi kd\,\varepsilon^{-1}
\sum_{j=1}^k
\int_{\mathbb R^{k-1}}
\sup_{x_j\in\mathbb R}|Q_j(x)\varrho(x)|\,d\hat{x}^j.
\end{equation}

Combining the estimates \eqref{eq-lem-est},
\eqref{eq-grad-term-est}, and \eqref{eq-det-term-est}, and taking into account that $0\le \eta_r\le 1$, we obtain
\begin{align*}
\biggl|
\int_{\mathbb R^k}(\partial_\theta\varphi)(f)\varrho\eta_r\,dx
\biggr|
&\le
\int_{\mathbb R^k}
I_{\{(\det J_f)^2\le 2\varepsilon^2\}}\varrho\,dx
\\
&
+
t\varepsilon^{-1}
\sum_{j=1}^k
\int_{\mathbb R^k}|\partial_{x_j}Q_j|\varrho\,dx
+
(2+18\pi kd)t\varepsilon^{-1}
\sum_{j=1}^k
\int_{\mathbb R^{k-1}}
\sup_{x_j\in\mathbb R}|Q_j(x)\varrho(x)|\,d\hat{x}^j
\\
&
+
r^{-1}\|\eta'\|_\infty t\varepsilon^{-1}
\int_{\mathbb R^k}|A_f\theta|\varrho\,dx.
\end{align*}
Letting $r\to\infty$ gives the claimed estimate. \end{proof}

\subsection{Dimension-dependent regularity estimate}

We will use the following two results from~\cite{Kos21}. The first one is a
Markov--Bernstein-type inequality for directional derivatives.

\begin{theorem}[{\cite[Theorem~3.1]{Kos21}}]\label{th-MB}
For each $d\in\mathbb N$ there exists a constant $C(d)>0$ such that,
for every $n\in\mathbb N$, every absolutely continuous log-concave measure
$\mu$ on $\mathbb R^n$ with density $\varrho$, every
$g\in\mathcal P_d(\mathbb R^n)$, and every unit vector
$h\in\mathbb R^n$, one has
\[
\|\partial_h g\|_{L^2(\mu)}
\le
C(d)
\biggl(
\int_{\langle h\rangle^\perp}
\sup_{s\in\mathbb R}\varrho(x+sh)\,dx
\biggr)
\|g\|_{L^2(\mu)}.
\]
\end{theorem}

\begin{remark}\label{rm-variance}
Since partial derivatives are unchanged when a constant is added to $g$, one may
apply Theorem~\ref{th-MB} to 
\[
g-\int_{\mathbb R^n}g\,d\mu.
\] 
Hence the
$L^2(\mu)$ norm on the right-hand side can be replaced by
$\sqrt{\operatorname{Var}_\mu(g)}$.
\end{remark}

The second result is a section estimate for polynomials with respect to
isotropic log-concave measures.

\begin{theorem}[{\cite[Lemma~3.2]{Kos21}}]\label{th-section}
For every $n,d\in\mathbb N$ there exists a constant $c(d,n)>0$ such that,
for every isotropic log-concave measure $\mu$ on $\mathbb R^n$ with density
$\varrho$, every $g\in\mathcal P_d(\mathbb R^n)$, and every unit vector
$h\in\mathbb R^n$, one has
\[
\int_{\langle h\rangle^\perp}
\sup_{s\in\mathbb R}
|g(x+sh)\varrho(x+sh)|\,dx
\le
c(d,n)\|g\|_{L^1(\mu)}.
\]
\end{theorem}
For $n\ge2$, this is Lemma~3.2 in~\cite{Kos21}. For $n=1$, it follows by
combining \cite[Corollary~2.4]{Klartag07}, \cite[Lemma~5.14(a),(d)]{LV07},
and the equivalence of all norms on the finite-dimensional space of
polynomials of degree at most $d$.

\begin{lemma}\label{lem-dimensional}
Let $n,k,d\in\mathbb N$, $n\ge k$, and $d\ge2$. There exists a constant
$C(k,d,n)>0$, depending only on $k,d,n$, such that, for every isotropic
log-concave measure $\mu$ on $\mathbb R^n$, every
$f=(f_1,\ldots,f_k)\in\mathcal P_d(\mathbb R^n;\mathbb R^k)$, every unit
vector $\theta\in\mathbb R^k$, every $t>0$, and every
$\varphi\in C_b^\infty(\mathbb R^k)$ satisfying
\[
\|\varphi\|_\infty\le t,
\quad
\|\partial_\theta\varphi\|_\infty\le 1,
\]
one has
\[
\biggl(\int_{\mathbb R^n}\Delta_f\,d\mu\biggr)^{\frac{1}{2(k(d-1)+1)}}
\biggl|
\int_{\mathbb R^n}(\partial_\theta\varphi)(f)\,d\mu
\biggr|
\le
C(k,d,n)t^{\frac{1}{k(d-1)+1}}
\biggl(\sum_{j=1}^k\operatorname{Var}_\mu(f_j)\biggr)^{\frac{k-1}{2(k(d-1)+1)}}.
\]
\end{lemma}

\begin{proof}
First, assume that $\mu$ has a density $\varrho\in C^\infty(\mathbb R^n)$.
	
{\bf Step 1.} 
Let $H$ be a $k$-dimensional subspace of $\mathbb R^n$, and let
$e_1,\ldots,e_k$ be an orthonormal basis in $H$. Define
\[
J_f^H
:=
\begin{pmatrix}
\langle \nabla f_1,e_1\rangle & \langle \nabla f_1,e_2\rangle & \ldots & \langle \nabla f_1,e_k\rangle \\
\langle \nabla f_2,e_1\rangle & \langle \nabla f_2,e_2\rangle & \ldots & \langle \nabla f_2,e_k\rangle \\
\ldots & \ldots & \ldots & \ldots \\
\langle \nabla f_k,e_1\rangle & \langle \nabla f_k,e_2\rangle & \ldots & \langle \nabla f_k,e_k\rangle
\end{pmatrix}
\]
and
\[
A_f^H:=\operatorname{adj}(J_f^H).
\]
For $j\in\{1,\ldots,k\}$, set
\[
Q_j^H(x):=\langle A_f^H(x)\theta,e_j\rangle.
\]

Write
\[
x=y+z,
\quad
y\in H,
\quad
z\in H^\perp.
\]
For each fixed $z\in H^\perp$, the function $y\mapsto\varrho(y+z)$ is log-concave and integrable on $H$. Hence
Lemma~\ref{key-lemma-2}, applied on the subspace $H$, gives
\begin{align*}
&
\biggl|\int_H(\partial_\theta\varphi)(f(y+z))\varrho(y+z)\,dy\biggr|
\\
&\le
\int_H
I_{\{(\det J_f^H)^2\le 2\varepsilon^2\}}(y+z)\varrho(y+z)\,dy
+
t\varepsilon^{-1}
\sum_{j=1}^k
\int_H|\partial_{e_j}Q_j^H(y+z)|\varrho(y+z)\,dy
\\
&
+
(2+18\pi kd)t\varepsilon^{-1}
\sum_{j=1}^k
\int_{H\cap\langle e_j\rangle^\perp}
\sup_{s\in\mathbb R}
|Q_j^H(\hat y^j+se_j+z)\varrho(\hat y^j+se_j+z)|\,d\hat y^j.
\end{align*}
Integrating with respect to
$z\in H^\perp$, we obtain
\begin{align*}
\biggl|\int_{\mathbb R^n}(\partial_\theta\varphi)(f)\,d\mu\biggr|
&\le
\mu\bigl((\det J_f^H)^2\le 2\varepsilon^2\bigr)
+
t\varepsilon^{-1}
\sum_{j=1}^k
\|\partial_{e_j}Q_j^H\|_{L^1(\mu)}
\\
&+
(2+18\pi kd)t\varepsilon^{-1}
\sum_{j=1}^k
\int_{\langle e_j\rangle^\perp}
\sup_{s\in\mathbb R}
|Q_j^H(x+se_j)\varrho(x+se_j)|\,dx.
\nonumber
\end{align*}

We now estimate the two last terms. By Theorem~\ref{th-MB}, applied to the
polynomial $Q_j^H$, and by Theorem~\ref{th-section} applied to the constant
polynomial $1$, one has
\[
\|\partial_{e_j}Q_j^H\|_{L^1(\mu)}
\le
\|\partial_{e_j}Q_j^H\|_{L^2(\mu)}
\le
C_1(k,d,n)\|Q_j^H\|_{L^2(\mu)}.
\]
Also, by Theorem~\ref{th-section},
\[
\int_{\langle e_j\rangle^\perp}
\sup_{s\in\mathbb R}
|Q_j^H(x+se_j)\varrho(x+se_j)|\,dx
\le
C_2(k,d,n)\|Q_j^H\|_{L^1(\mu)}
\le
C_2(k,d,n)\|Q_j^H\|_{L^2(\mu)}.
\]
Therefore,
\[
\biggl|\int_{\mathbb R^n}(\partial_\theta\varphi)(f)\,d\mu\biggr|
\le
\mu\bigl((\det J_f^H)^2\le 2\varepsilon^2\bigr)+
C_3(k,d,n)t\varepsilon^{-1}
\sum_{j=1}^k\|Q_j^H\|_{L^2(\mu)}.
\]
We have
\[
\sum_{j=1}^k\|Q_j^H\|_{L^2(\mu)}
\le
\sqrt{k}\,\||A_f^H\theta|\|_{L^2(\mu)}
\le
\sqrt{k}\,\|\|A_f^H\|_{\rm op}\|_{L^2(\mu)},
\]
and, since
\[
\|A_f^H\|_{\rm op}
\le
C_4(k)
\biggl(\sum_{j=1}^k|\nabla f_j|^2\biggr)^{\frac{k-1}{2}},
\]
we also have
\[
\|\|A_f^H\|_{\rm op}\|_{L^2(\mu)}
\le
C_4(k)
\biggl\|
\sum_{j=1}^k|\nabla f_j|^2
\biggr\|_{L^{k-1}(\mu)}^{\frac{k-1}{2}}.
\]
By the equivalence of moments on spaces of polynomials of bounded degree
\cite{Bobkov00}, it follows that
\[
\biggl\|
\sum_{j=1}^k|\nabla f_j|^2\biggr\|_{L^{k-1}(\mu)}
\le
C_5(k,d)
\sum_{j=1}^k\sum_{i=1}^n\|\partial_{x_i} f_j\|_{L^2(\mu)}^2.
\]
Finally, Theorem~\ref{th-MB}, Remark~\ref{rm-variance}, and
Theorem~\ref{th-section} applied to the constant polynomial $1$ give
\[
\|\partial_{x_i} f_j\|_{L^2(\mu)}^2
\le
C_6(d,n)\operatorname{Var}_\mu (f_j)
\quad \forall j\in\{1,\ldots,k\}
\quad \forall i\in\{1,\ldots,n\}.
\]
Therefore,
\[
\|\|A_f^H\|_{\rm op}\|_{L^2(\mu)}
\le
C_7(k,d,n)
\biggl(\sum_{j=1}^k\operatorname{Var}_\mu (f_j)\biggr)^{\frac{k-1}{2}}.
\]
Thus,
\begin{equation}\label{eq-H-estimate-2}
\biggl|\int_{\mathbb R^n}(\partial_\theta\varphi)(f)\,d\mu\biggr|
\le
\mu\bigl((\det J_f^H)^2\le 2\varepsilon^2\bigr)+
C_8(k,d,n)t\varepsilon^{-1}
\biggl(\sum_{j=1}^k\operatorname{Var}_\mu (f_j)\biggr)^{\frac{k-1}{2}}.
\end{equation}

{\bf Step 2.}
Assume first that
\[
\int_{\mathbb R^n}\bigl(\det J_f^H\bigr)^2\,d\mu>0.
\]
By the
Carbery--Wright inequality \eqref{CW-est} applied to the polynomial
$\det J_f^H\in \mathcal{P}_{k(d-1)}(\mathbb{R}^n)$,
we obtain
\[
\mu\bigl((\det J_f^H)^2\le 2\varepsilon^2\bigr)
\le
C_9(k,d)\varepsilon^{\frac{1}{k(d-1)}}\biggl(\int_{\mathbb R^n}\bigl(\det J_f^H\bigr)^2\,d\mu\biggr)^{-\frac{1}{2k(d-1)}}.
\]
Hence \eqref{eq-H-estimate-2} implies
\begin{align*}
\biggl|\int_{\mathbb R^n}(\partial_\theta\varphi)(f)\,d\mu\biggr|
&\le
C_9(k,d)\varepsilon^{\frac{1}{k(d-1)}}\biggl(\int_{\mathbb R^n}\bigl(\det J_f^H\bigr)^2\,d\mu\biggr)^{-\frac{1}{2k(d-1)}}
\\
&\qquad\qquad\qquad\qquad\qquad\qquad\qquad+
C_8(k,d,n)t\varepsilon^{-1}
\biggl(\sum_{j=1}^k\operatorname{Var}_\mu (f_j)\biggr)^{\frac{k-1}{2}}.
\end{align*}
Choosing
\[
\varepsilon
=
\biggl(
t
\biggl(\sum_{j=1}^k\operatorname{Var}_\mu (f_j)\biggr)^{\frac{k-1}{2}}
\Bigl(\int_{\mathbb R^n}(\det J_f^H)^2\,d\mu\Bigr)^{\frac{1}{2k(d-1)}}
\biggr)^{\frac{k(d-1)}{k(d-1)+1}},
\]
we obtain
\[
\biggl|
\int_{\mathbb R^n}(\partial_\theta\varphi)(f)\,d\mu
\biggr|
\le
C_{10}(k,d,n)
t^{\frac{1}{k(d-1)+1}}
\biggl(\sum_{j=1}^k\operatorname{Var}_\mu (f_j)\biggr)^{\frac{k-1}{2(k(d-1)+1)}}
\Bigl(\int_{\mathbb R^n}(\det J_f^H)^2\,d\mu\Bigr)^{
-\frac{1}{2(k(d-1)+1)}}.
\]
Equivalently,
\[
\biggl(\int_{\mathbb R^n}(\det J_f^H)^2\,d\mu\biggr)
\biggl|
\int_{\mathbb R^n}(\partial_\theta\varphi)(f)\,d\mu
\biggr|^{2(k(d-1)+1)}
\le
C_{11}(k,d,n)t^2
\biggl(\sum_{j=1}^k\operatorname{Var}_\mu (f_j)\biggr)^{k-1}.
\]
This estimate is also valid when the first integral on the left-hand side is zero.

Averaging the last estimate over all $H\in Gr(k,n)$ with respect to the
normalized Haar measure $\nu_{k,n}$ and applying Lemma~\ref{lem-Gr-average},
we obtain
\[
\biggl(\int_{\mathbb R^n}\Delta_f\,d\mu\biggr)
\biggl|
\int_{\mathbb R^n}(\partial_\theta\varphi)(f)\,d\mu
\biggr|^{2(k(d-1)+1)}
\le
C_{12}(k,d,n)t^2
\biggl(\sum_{j=1}^k\operatorname{Var}_\mu (f_j)\biggr)^{k-1}.
\]
Raising both sides to the power $\frac{1}{2(k(d-1)+1)}$ gives the claimed
estimate in the case where $\mu$ has a smooth density.

{\bf Step 3.}
It remains to remove the smoothness assumption. Let $X$ be a random vector
with distribution $\mu$, and let $G$ be an independent standard Gaussian
vector in $\mathbb R^n$. For $\sigma>0$, set
\[
X_\sigma:=\frac{X+\sigma G}{\sqrt{1+\sigma^2}},
\]
and denote by $\mu_\sigma$ the distribution of $X_\sigma$. Then
$\mu_\sigma$ has a $C^\infty$ density, is log-concave, and is isotropic.
Hence the estimate already proved for smooth densities applies to
$\mu_\sigma$:
\[
\biggl(\int_{\mathbb R^n}\Delta_f\,d\mu_\sigma\biggr)^{\frac{1}{2(k(d-1)+1)}}
\biggl|
\int_{\mathbb R^n}(\partial_\theta\varphi)(f)\,d\mu_\sigma
\biggr|
\le
C(k,d,n)t^{\frac{1}{k(d-1)+1}}
\biggl(\sum_{j=1}^k\operatorname{Var}_{\mu_\sigma} (f_j)\biggr)^{\frac{k-1}{2(k(d-1)+1)}}.
\]
We now let $\sigma\to0$. For every polynomial $Q$, expanding
\[
Q\biggl(\frac{X+\sigma G}{\sqrt{1+\sigma^2}}\biggr)
\]
and using the finiteness of all polynomial moments, we obtain
\[
\int_{\mathbb R^n}Q\,d\mu_\sigma
=
\mathbb E Q(X_\sigma)
\to
\mathbb E Q(X)
=
\int_{\mathbb R^n}Q\,d\mu.
\]
Therefore,
\[
\int_{\mathbb R^n}\Delta_f\,d\mu_\sigma
\to
\int_{\mathbb R^n}\Delta_f\,d\mu
\]
and
\[
\operatorname{Var}_{\mu_\sigma}(f_j)
\to
\operatorname{Var}_{\mu}(f_j)
\quad \forall j\in\{1,\ldots,k\}.
\]
Finally, since $X_\sigma\to X$ almost surely and
$(\partial_\theta\varphi)(f)$ is bounded and continuous, the Lebesgue dominated convergence theorem gives 
\[
\int_{\mathbb R^n}(\partial_\theta\varphi)(f)\,d\mu_\sigma
\to
\int_{\mathbb R^n}(\partial_\theta\varphi)(f)\,d\mu.
\]
Passing to the limit in the preceding estimate gives the claimed inequality
for the original isotropic log-concave measure $\mu$.
\end{proof}

\section{Control of nondegeneracy by the covariance of the extended mapping}
\label{sect-algeb}

Here we show that, for fixed $n,d,k\in\mathbb N$, nondegeneracy of a mapping
$f\in\mathcal P_d(\mathbb R^n;\mathbb R^k)$ in terms of $\Delta_f$ is equivalent
to nondegeneracy of the covariance matrix of the extended mapping~$F_\mu^d(f)$.

\subsection{Algebraic dependence}

\begin{lemma}\label{lem-ann-deg}
Let $n,d,k\in\mathbb N$, $n\ge k$, and let
$f=(f_1,\ldots,f_k)\in\mathcal P_d(\mathbb R^n;\mathbb{R}^k)$
be a polynomial mapping with nonconstant components. Assume that all
$k\times k$ minors of the matrix $J_f$ are zero polynomials. Then there
exists a nonzero polynomial
$Q\in\mathcal P_{d^{k-1}}(\mathbb R^k)$
such that
\[
Q(f_1,\ldots,f_k)\equiv0.
\]
\end{lemma}

\begin{proof}
Since all $k\times k$ minors of the matrix $J_f$ are zero polynomials, the
Jacobian criterion for algebraic independence implies that
$f_1,\ldots,f_k$ are algebraically dependent
(see, for instance, \cite{EhrenborgRota} or \cite[Theorem~2]{Kayal09}).

Set
\[
d_j=\deg f_j
\quad \forall j\in\{1,\ldots,k\}.
\]
Since the polynomials $f_j$ are nonconstant and have degree at most $d$,
\[
1\le d_j\le d,
\quad \forall j\in\{1,\ldots,k\}.
\]

We apply P{\l}oski's form of Perron's theorem
\cite[Corollary~(1.6)]{Ploski86}. It gives a nonzero polynomial
\[
R(t)=\sum_{{\bf m}}a({\bf m}) t^{\bf m},
\quad
t^{\bf m}:=t_1^{m_1}\ldots t_k^{m_k},
\quad a({\bf m})\in \mathbb{C},
\]
such that
\[
R(f_1,\ldots,f_k)\equiv0
\]
and
\[
\max_{a({\bf m})\ne0}
\sum_{j=1}^k m_j d_j
\le
d_1\ldots d_k.
\]
Hence, for every ${\bf m}$ such that $a({\bf m})\ne0$,
\[
\min\{d_1,\ldots,d_k\}|{\bf m}|
\le
\sum_{j=1}^k m_j d_j
\le
d_1\ldots d_k.
\]
Therefore,
\[
|{\bf m}|
\le
\frac{d_1\ldots d_k}{\min\{d_1,\ldots,d_k\}}
\le d^{k-1}.
\]
Thus,
\[
\deg R\le d^{k-1}.
\]

Now write
\[
R=R_1+iR_2,
\quad
R_1,R_2\in \mathcal{P}_{d^{k-1}}(\mathbb{R}^k).
\]
Since the polynomials $f_j$ have real coefficients and
\[
R(f_1,\ldots,f_k)\equiv0,
\]
we have
\[
R_1(f_1,\ldots,f_k)\equiv0
\quad\text{and}\quad
R_2(f_1,\ldots,f_k)\equiv0.
\]
At least one of $R_1,R_2$ is nonzero. Taking this nonzero real or imaginary
part, we obtain a nonzero polynomial
$Q\in\mathcal P_{d^{k-1}}(\mathbb R^k)$
with
\[
Q(f_1,\ldots,f_k)\equiv0.
\]
This proves the lemma.
\end{proof}

\begin{lemma}\label{lem-n-less-k}
Let $d,n,k\in\mathbb N$, $n<k$, and let $\mu$ be an absolutely continuous
log-concave measure on $\mathbb R^n$. Let
$f=(f_1,\ldots,f_k)\in\mathcal P_d(\mathbb R^n;\mathbb R^k)$
be such that
\[
\operatorname{Var}_\mu(f_j)>0
\quad
\forall j\in\{1,\ldots,k\}.
\]
Then
$\det \Sigma_\mu^d(f)=0$.
\end{lemma}

\begin{proof}
Since
\[
\operatorname{Var}_\mu(f_j)>0,
\]
the polynomials $\hat f_j$ are well defined and nonconstant. Since $n<k$,
we may choose $n+1$ components, say
$\hat f_1,\ldots,\hat f_{n+1}$.
These are nonconstant polynomials on $\mathbb R^n$ of degree at most $d$.

By the classical Perron theorem for $n+1$ polynomials in $n$ variables
(see, for instance, \cite[Theorem~1.1]{Ploski05}), there exists a nonzero polynomial
$R\in\mathcal P_{d^n}(\mathbb R^{n+1})$
such that
\[
R(\hat f_1,\ldots,\hat f_{n+1})\equiv0.
\]
Define
\[
Q(t_1,\ldots,t_k)=R(t_1,\ldots,t_{n+1}).
\]
Then $Q$ is a nonzero polynomial on $\mathbb R^k$,
\[
Q(\hat f_1,\ldots,\hat f_k)\equiv0,
\]
and, since $d^n\le d^{k-1}$,
$Q\in \mathcal P_{d^{k-1}}(\mathbb R^k)$.

Write
\[
Q(t)=q(0)+\sum_{1\le|\mathbf m|\le d^{k-1}}q(\mathbf m)t^{\mathbf m}.
\]
The vector
$\bigl(q(\mathbf m)\bigr)_{1\le|\mathbf m|\le d^{k-1}}$
is nonzero, since otherwise $Q$ would be a nonzero constant polynomial, which
cannot vanish after substitution. Hence
\[
\sum_{1\le|\mathbf m|\le d^{k-1}}q(\mathbf m)\hat f^{\mathbf m}
\equiv
-q(0).
\]
Thus, a nontrivial linear combination of the components of $F_\mu^d(f)$ is
constant. Therefore,
\[
\operatorname{Var}_\mu
\biggl(
\sum_{1\le|\mathbf m|\le d^{k-1}}q(\mathbf m)\hat f^{\mathbf m}
\biggr)=0,
\]
and hence
\[
\det \Sigma_\mu^d(f)=0.
\]
This proves the lemma.
\end{proof}

\begin{lemma}\label{lem-zero-equiv}
Let $n,d,k\in\mathbb N$, $n\ge k$, and let $\mu$ be an isotropic
log-concave measure on $\mathbb R^n$. Let
$f=(f_1,\ldots,f_k)\in \mathcal P_d(\mathbb R^n;\mathbb R^k)$
and assume that
\[
\operatorname{Var}_\mu(f_j)>0,
\quad \forall j\in\{1,\ldots,k\}.
\]
Then
\[
\int_{\mathbb R^n}\Delta_{\hat f}\,d\mu=0
\quad\Longleftrightarrow\quad
\det \Sigma_\mu^d(f)=0.
\]
\end{lemma}

\begin{proof}
Assume first that
\[
\int_{\mathbb R^n}\Delta_{\hat f}\,d\mu=0.
\]
Since $\Delta_{\hat f}\ge0$, we have
\[
\Delta_{\hat f}=0
\quad \mu\text{-a.e.}
\]
The measure $\mu$ is isotropic and log-concave, hence its support has
nonempty interior in $\mathbb R^n$. Since $\Delta_{\hat f}$ is a polynomial,
it follows that
\[
\Delta_{\hat f}\equiv0
\quad\text{on } \mathbb R^n.
\]
By the Cauchy--Binet formula,
\[
\Delta_{\hat f}
=
\sum_{1\le i_1<\ldots<i_k\le n}
\Bigl(\det
\Bigl(\partial_{x_{i_s}} \hat f_r
\Bigr)_{1\le r,s\le k}
\Bigr)^2.
\]
Therefore all $k\times k$ minors of $J_{\hat f}$ are zero polynomials. By Lemma \ref{lem-ann-deg}, there exists a nonzero polynomial
$Q\in \mathcal{P}_{d^{k-1}}(\mathbb{R}^k)$
such that
\[
Q(\hat f_1,\ldots,\hat f_k)\equiv0.
\]
Write
\[
Q(t)=q(0)+
\sum_{1\le |\mathbf m|\le d^{k-1}}
q(\mathbf m)t^{\mathbf m}.
\]
The vector
$\bigl(q(\mathbf m)\bigr)_{1\le |\mathbf m|\le d^{k-1}}$
is nonzero. Hence,
\[
\sum_{1\le |\mathbf m|\le d^{k-1}}
q(\mathbf m)\hat f^{\mathbf m}
\equiv -q(0).
\]
Thus, a nontrivial linear combination of the components of $F_\mu^d(f)$ is
constant. Therefore,
\[
\operatorname{Var}_\mu
\biggl(
\sum_{1\le |\mathbf m|\le d^{k-1}}
q(\mathbf m)\hat f^{\mathbf m}
\biggr)=0,
\]
and hence
\[
\det \Sigma_\mu^d(f)=0.
\]

Conversely, assume that
\[
\det \Sigma_\mu^d(f)=0.
\]
Then there exists a nonzero vector
$\bigl(q(\mathbf m)\bigr)_{1\le |\mathbf m|\le d^{k-1}}$
such that
\[
\operatorname{Var}_\mu
\biggl(
\sum_{1\le |\mathbf m|\le d^{k-1}}
q(\mathbf m)\hat f^{\mathbf m}
\biggr)=0.
\]
Therefore, there exists a constant $q(0)\in\mathbb R$ such that
\[
\sum_{1\le |\mathbf m|\le d^{k-1}}
q(\mathbf m)\hat f^{\mathbf m}
=
-q(0)
\quad \mu\text{-a.e.}
\]
Since $\mu$ has support with nonempty interior and the left-hand side is a
polynomial, this equality holds identically on $\mathbb R^n$. Hence
\[
Q(\hat f_1,\ldots,\hat f_k)\equiv0,
\]
where
\[
Q(t)=q(0)+
\sum_{1\le |\mathbf m|\le d^{k-1}}
q(\mathbf m)t^{\mathbf m}.
\]
It is clear that the polynomial $Q$ is nonzero.
Thus, $\hat f_1,\ldots,\hat f_k$ are algebraically dependent. By the
Jacobian criterion for algebraic independence (see, for instance, \cite{EhrenborgRota}), all
$k\times k$ minors of $J_{\hat f}$ are zero polynomials. Therefore, by
Cauchy--Binet,
$\Delta_{\hat f}\equiv0$.
Consequently,
\[
\int_{\mathbb R^n}\Delta_{\hat f}\,d\mu=0.
\]
This proves the stated equivalence.
\end{proof}

\subsection{Equivalence of quantitative nondegeneracy conditions in fixed dimension}

We begin with the following compactness lemma.

\begin{lemma}\label{lem-compactness}
Let $n,d,k\in \mathbb{N}$ and let
$(\mu_r, f_{1,r},\ldots,f_{k,r})$ be a sequence such that each $\mu_r$ is an
isotropic log-concave measure on $\mathbb{R}^n$,
\[
f_{j,r}\in \mathcal{P}_d(\mathbb{R}^n),
\quad
\forall j\in\{1,\ldots,k\},\ \forall r\in\mathbb{N},
\]
and
\[
\int_{\mathbb{R}^n}f_{j,r}\,d\mu_r =0,\quad
\int_{\mathbb{R}^n}f_{j,r}^2\,d\mu_r =1,
\quad
\forall j\in\{1,\ldots,k\},\ \forall r\in\mathbb{N}.
\]
Then there exist an isotropic log-concave measure $\mu$, polynomials
$f_1,\ldots,f_k\in\mathcal P_d(\mathbb R^n)$, and a subsequence
$\{r_\ell\}_{\ell\in \mathbb{N}}$ such that
$\mu_{r_\ell}\Rightarrow\mu$
weakly, and, for every $j\in\{1,\ldots,k\}$, the coefficients of
$f_{j,r_\ell}$ converge to the coefficients of $f_j$. Moreover,
\[
\int_{\mathbb{R}^n}f_j\,d\mu =0,\quad
\int_{\mathbb{R}^n}f_j^2\,d\mu =1,
\quad
\forall j\in\{1,\ldots,k\}.
\]
\end{lemma}

\begin{proof}
Since each $\mu_r$ is isotropic, we have
\begin{equation}\label{eq-isotr}
\int_{\mathbb R^n}|x|^2\,\mu_r(dx)=n.
\end{equation}
Therefore, the sequence $\{\mu_r\}_{r\in\mathbb N}$ is tight. Passing to a
subsequence, we may assume that
$\mu_r\Rightarrow \mu$
weakly for some probability measure $\mu$ on $\mathbb R^n$.

The class of log-concave measures is closed under weak convergence. Moreover,
by the moment comparison estimate for log-concave measures
\cite[Theorem~5.22]{LV07} and by \eqref{eq-isotr}, the moments of every fixed
order are uniformly bounded. Thus, weak convergence implies convergence of all
polynomial moments (see \cite[Theorem~3.5]{Billingsley99}), that is,
\begin{equation}\label{eq-monom-conv}
\int_{\mathbb{R}^n} x^{\bf m}\,\mu_r(dx)\to
\int_{\mathbb{R}^n} x^{\bf m}\,\mu(dx)
\quad
\forall {\bf m}\in \mathbb{Z}_+^n.
\end{equation}
In particular, $\mu$ is also an isotropic log-concave measure.

It remains to prove that the coefficients of the polynomials $f_{j,r}$ are
bounded. By \cite[Theorem~5.14]{LV07}, there exists a number $c_1(n)>0$,
depending only on $n$, such that every isotropic log-concave density
$\varrho$ satisfies
\[
\varrho(x)\ge c_1(n) I_{\{|x|\le 1/9\}}.
\]
Consequently, for all $j\in\{1,\ldots,k\}$ and all $r\in\mathbb N$,
\[
1=\int_{\mathbb{R}^n}f_{j,r}^2\,d\mu_r
\ge
c_1(n)\int_{\{|x|\le 1/9\}}|f_{j,r}(x)|^2\,dx.
\]

Now, the mappings
\[
g\mapsto
\biggl(\int_{\{|x|\le1/9\}}|g(x)|^2\,dx\biggr)^{1/2}
\]
and
\[
g\mapsto \max_{|{\bf m}|\le d}|a({\bf m})|,
\quad
g=\sum_{|\mathbf m|\le d} a(\mathbf m)x^{\mathbf m},
\]
define norms on the finite-dimensional space $\mathcal{P}_d(\mathbb{R}^n)$.
Thus, there exists a constant $c_2(n,d)>0$ such that
\[
\int_{\{|x|\le1/9\}}|g(x)|^2\,dx
\ge
c_2(n,d)
\max_{|{\bf m}|\le d}|a({\bf m})|^2
\quad
\forall g\in \mathcal{P}_d(\mathbb{R}^n).
\]
Applying this to $g=f_{j,r}$, we get
\[
\max_{|{\bf m}|\le d}|a_{j,r}({\bf m})|^2
\le
\frac{1}{c_1(n)c_2(n,d)},
\]
where
\[
f_{j,r}(x)=\sum_{|\mathbf m|\le d}a_{j,r}({\bf m})x^{\bf m}.
\]
Thus, all coefficients of all polynomials $f_{j,r}$ are uniformly bounded.
Hence, passing to a further subsequence, we may assume that, for every
$j\in\{1,\ldots,k\}$, the polynomials $f_{j,r}$ converge coefficientwise to
some polynomial
\[
f_j\in\mathcal P_d(\mathbb R^n).
\]
Renaming this subsequence as $\{r_\ell\}_{\ell\in\mathbb N}$, we have
\[
\mu_{r_\ell}\Rightarrow\mu
\]
weakly, and the coefficients of each $f_{j,r_\ell}$ converge to the
coefficients of $f_j$.

Finally, taking this coefficientwise convergence together with
\eqref{eq-monom-conv} into account, we obtain
\[
\int_{\mathbb{R}^n}f_j\,d\mu=0,\quad
\int_{\mathbb{R}^n}f_j^2\,d\mu=1,
\quad
\forall j\in\{1,\ldots,k\}.
\]
This proves the stated compactness assertion.
\end{proof}

\begin{theorem}\label{th-compactness}
Let $n,d,k\in\mathbb N$ and assume that $n\ge k$. There exist two functions
\[
\alpha_{n,d,k}\colon(0,\infty)\to(0,\infty)
\quad\text{and}\quad
\beta_{n,d,k}\colon(0,\infty)\to(0,\infty),
\]
with the following property.

Let $\mu$ be an isotropic log-concave measure on $\mathbb R^n$,
and let
$f=(f_1,\ldots,f_k)\in \mathcal P_d(\mathbb R^n;\mathbb R^k)$
be such that
\[
\operatorname{Var}_\mu(f_j)>0,
\quad
\forall j\in\{1,\ldots,k\}.
\]
Then, for every $a>0$,
\[
\int_{\mathbb{R}^n} \Delta_{\hat{f}}\,d\mu\ge a
\quad\Longrightarrow\quad
\det \Sigma_\mu^d(f)\ge \beta_{n,d,k}(a),
\]
and, for every $b>0$,
\[
\det \Sigma_\mu^d(f)\ge b
\quad\Longrightarrow\quad
\int_{\mathbb{R}^n} \Delta_{\hat{f}}\,d\mu\ge \alpha_{n,d,k}(b).
\]
\end{theorem}

\begin{proof}
We prove the two implications by contradiction.

First, fix $a>0$. Suppose that the first implication is false. Then there
exist isotropic log-concave measures $\mu_r$ on $\mathbb R^n$ and mappings
$f_r=(f_{1,r},\ldots,f_{k,r})\in\mathcal P_d(\mathbb R^n;\mathbb R^k)$
such that $\operatorname{Var}_\mu(f_{j,r})>0$ for all $j\in\{1,\ldots,k\}$,
\[
\int_{\mathbb R^n}\Delta_{\hat f_r}\,d\mu_r\ge a,
\quad
\text{and}
\quad
\det \Sigma_{\mu_r}^d(f_r)\to0.
\]
Set
\[
g_{j,r}=\hat f_{j,r},
\quad \forall j\in\{1,\ldots,k\}.
\]
Then
\[
\int_{\mathbb R^n}g_{j,r}\,d\mu_r=0,\quad
\int_{\mathbb R^n}g_{j,r}^2\,d\mu_r=1,
\quad
\forall j\in\{1,\ldots,k\}.
\]
By Lemma~\ref{lem-compactness}, passing to a subsequence, we may assume that
$\mu_r\Rightarrow\mu$
weakly, where $\mu$ is an isotropic log-concave measure, and that, for every
$j\in\{1,\ldots,k\}$, the coefficients of $g_{j,r}$ converge to the
coefficients of some polynomial
$g_j\in\mathcal P_d(\mathbb R^n)$.
Moreover,
\[
\int_{\mathbb R^n}g_j\,d\mu=0,\quad
\int_{\mathbb R^n}g_j^2\,d\mu=1,
\quad
\forall j\in\{1,\ldots,k\}.
\]
Let
$g=(g_1,\ldots,g_k)$.
By \eqref{eq-monom-conv} and by the coefficientwise convergence,
\begin{equation}\label{eq-conv-1}
\int_{\mathbb R^n}Q(g)\,d\mu
=
\lim_{r\to\infty}
\int_{\mathbb R^n}Q(g_r)\,d\mu_r
\end{equation}
and
\begin{equation}\label{eq-conv-2}
\int_{\mathbb R^n}U(\nabla g_1,\ldots,\nabla g_k)\,d\mu
=
\lim_{r\to\infty}
\int_{\mathbb R^n}U(\nabla g_{1,r},\ldots,\nabla g_{k,r})\,d\mu_r
\end{equation}
for every pair of polynomials
$Q\in \mathcal{P}_{D_1}(\mathbb{R}^k)$ and
$U\in \mathcal{P}_{D_2}(\mathbb{R}^{kn})$.
In particular, this implies that
\[
\int_{\mathbb R^n}\Delta_g\,d\mu
=
\lim_{r\to\infty}
\int_{\mathbb R^n}\Delta_{\hat f_r}\,d\mu_r
\ge a,
\]
and
\[
\det \Sigma_\mu^d(g)
=
\lim_{r\to\infty}\det \Sigma_{\mu_r}^d(f_r)
=
0.
\]
This contradicts Lemma~\ref{lem-zero-equiv}. Therefore, for every $a>0$, there exists $\beta_{n,d,k}(a)>0$ such that
\[
\int_{\mathbb R^n}\Delta_{\hat f}\,d\mu\ge a
\quad\Longrightarrow\quad
\det \Sigma_\mu^d(f)\ge \beta_{n,d,k}(a).
\]

Now fix $b>0$. Suppose that the second implication is false. Then there
exist isotropic log-concave measures $\mu_r$ on $\mathbb R^n$ and mappings
$f_r=(f_{1,r},\ldots,f_{k,r})\in\mathcal P_d(\mathbb R^n;\mathbb R^k)$
such that $\operatorname{Var}_\mu(f_{j,r})>0$ for all $j\in\{1,\ldots,k\}$,
\[
\det \Sigma_{\mu_r}^d(f_r)\ge b,
\quad
\text{and}
\quad
\int_{\mathbb R^n}\Delta_{\hat f_r}\,d\mu_r\to0.
\]
Again set
\[
g_{j,r}=\hat f_{j,r},
\quad \forall j\in\{1,\ldots,k\}.
\]
Then
\[
\int_{\mathbb R^n}g_{j,r}\,d\mu_r=0,\quad
\int_{\mathbb R^n}g_{j,r}^2\,d\mu_r=1,
\quad
\forall j\in\{1,\ldots,k\}.
\]
By Lemma~\ref{lem-compactness}, passing to a subsequence, we may assume that
$\mu_r\Rightarrow\mu$
weakly, where $\mu$ is an isotropic log-concave measure, and that, for every
$j\in\{1,\ldots,k\}$, the coefficients of $g_{j,r}$ converge to the
coefficients of some polynomial
$g_j\in\mathcal P_d(\mathbb R^n)$.
As before, by \eqref{eq-conv-1} and \eqref{eq-conv-2},
\[
\int_{\mathbb R^n}\Delta_g\,d\mu
=
\lim_{r\to\infty}
\int_{\mathbb R^n}\Delta_{\hat f_r}\,d\mu_r
=
0,
\]
and
\[
\det \Sigma_\mu^d(g)
=
\lim_{r\to\infty}\det \Sigma_{\mu_r}^d(f_r)
\ge b,
\]
where $g=(g_1,\ldots,g_k)$.
This contradicts Lemma~\ref{lem-zero-equiv}.
Therefore, for every $b>0$, there exists $\alpha_{n,d,k}(b)>0$ such that
\[
\det \Sigma_\mu^d(f)\ge b
\quad\Longrightarrow\quad
\int_{\mathbb R^n}\Delta_{\hat f}\,d\mu\ge \alpha_{n,d,k}(b).
\]
This proves the theorem.
\end{proof}

\section{Proof of the main theorems}
\label{sect-main}

We start with the following dimension-dependent version of
Theorems~\ref{th-main-lorentz} and~\ref{th-main-besov}.

\begin{lemma}\label{lem-dimensional-2}
Let $\eta>0$, $n,k,d\in\mathbb N$, $n\ge k$, and $d\ge2$. There exists a
constant $C(\eta,k,d,n)>0$, depending only on $\eta,k,d,n$, such that, for
every absolutely continuous log-concave measure $\mu$ on $\mathbb R^n$ and
every
$f=(f_1,\ldots,f_k)\in\mathcal P_d(\mathbb R^n;\mathbb R^k)$
such that
\[
\sigma_{f_j}^2:=\operatorname{Var}_{\mu}(f_j)>0
\quad
\forall j\in\{1,\ldots,k\},
\]
and
\[
\det \Sigma_\mu^d(f)\ge \eta,
\]
the following two estimates hold.

First, for every $\psi\in C_0^\infty(\mathbb R^k)$ satisfying $0\le\psi\le1$,
one has
\[
\bigl(\sigma_{f_1}^2\cdot\ldots\cdot\sigma_{f_k}^2\bigr)^{\frac{1}{2(k(d-1)+1)}}
\int_{\mathbb R^n}\psi(f)\,d\mu
\le
C(\eta,k,d,n)
\|\psi\|_{L^1(\mathbb R^k)}^{\frac{1}{k(d-1)+1}}.
\]
Second, for every unit vector $\theta\in\mathbb R^k$, every $t>0$, and every
$\varphi\in C_b^\infty(\mathbb R^k)$ satisfying
\[
\|\varphi\|_\infty\le t,
\quad
\|\partial_\theta\varphi\|_\infty\le 1,
\]
one has
\[
\bigl(\sigma_{f_1}^2\cdot\ldots\cdot\sigma_{f_k}^2\bigr)^{\frac{1}{2(k(d-1)+1)}}
\biggl|
\int_{\mathbb R^n}(\partial_\theta\varphi)(f)\,d\mu
\biggr|
\le
C(\eta,k,d,n)t^{\frac{1}{k(d-1)+1}}
\biggl(\sum_{j=1}^k\sigma_{f_j}^2\biggr)^{\frac{k-1}{2(k(d-1)+1)}}.
\]
\end{lemma}

\begin{proof}
Let $m_\mu$ be the barycenter of $\mu$, and let $\Sigma_\mu$ be its covariance
matrix. Since $\mu$ is absolutely continuous, $\Sigma_\mu$ is positive
definite. Let $\nu$ be the image of $\mu$ under the affine map
\[
x\mapsto \Sigma_\mu^{-1/2}(x-m_\mu).
\]
Then $\nu$ is an isotropic log-concave measure on $\mathbb R^n$.

Let
\[
g_j(y)=f_j(m_\mu+\Sigma_\mu^{1/2}y)
\quad
\forall j\in\{1,\ldots,k\},
\]
and set
$g=(g_1,\ldots,g_k)$.
Then
$g\in\mathcal P_d(\mathbb R^n;\mathbb R^k)$,
and the distributions of $g$ under $\nu$ and of $f$ under $\mu$ coincide.
In particular,
\[
\operatorname{Var}_\nu(g_j)=\operatorname{Var}_\mu(f_j)=\sigma_{f_j}^2
\quad
\forall j\in\{1,\ldots,k\}.
\]
With the normalizations taken with respect to $\nu$ and $\mu$, respectively,
we have
\[
\widehat g_j(y)=\widehat f_j(m_\mu+\Sigma_\mu^{1/2}y),
\]
and hence
\[
\Sigma_\nu^d(g)=\Sigma_\mu^d(f).
\]
Thus,
\[
\det \Sigma_\nu^d(g)\ge\eta.
\]
By Theorem~\ref{th-compactness} applied to $\nu$ and $g$, there exists
$\alpha_{n,d,k}(\eta)>0$ such that
\[
\int_{\mathbb R^n}\Delta_{\widehat g}\,d\nu
\ge
\alpha_{n,d,k}(\eta).
\]
Since
\[
g_j=\sigma_{f_j}\widehat g_j+\int_{\mathbb R^n}g_j\,d\nu,
\]
we have
\[
\nabla g_j=\sigma_{f_j}\nabla\widehat g_j.
\]
Therefore,
\[
\Delta_g
=
\bigl(\sigma_{f_1}^2\cdot\ldots\cdot\sigma_{f_k}^2\bigr)
\Delta_{\widehat g},
\]
and consequently
\begin{equation}\label{eq-det-lower-est}
\int_{\mathbb R^n}\Delta_g\,d\nu
\ge
\alpha_{n,d,k}(\eta)
\bigl(\sigma_{f_1}^2\cdot\ldots\cdot\sigma_{f_k}^2\bigr).
\end{equation}

Applying Lemma~\ref{lem-lorentz-dimensional} to $\nu$ and $g$, and using
\eqref{eq-det-lower-est}, we obtain
\[
\bigl(\sigma_{f_1}^2\cdot\ldots\cdot\sigma_{f_k}^2\bigr)^{\frac{1}{2(k(d-1)+1)}}
\int_{\mathbb R^n}\psi(g)\,d\nu
\le
C_1(\eta,k,d,n)
\|\psi\|_{L^1(\mathbb R^k)}^{\frac{1}{k(d-1)+1}},
\]
where
\[
C_1(\eta,k,d,n)
=
C_1(k,d,n)\alpha_{n,d,k}(\eta)^{-\frac{1}{2(k(d-1)+1)}}
\]
and $C_1(k,d,n)$ is the constant from Lemma~\ref{lem-lorentz-dimensional}.
Since
\[
\int_{\mathbb R^n}\psi(g)\,d\nu
=
\int_{\mathbb R^n}\psi(f)\,d\mu,
\]
the first estimate follows.

Applying Lemma~\ref{lem-dimensional} to $\nu$ and $g$, and using
\eqref{eq-det-lower-est}, we obtain
\[
\bigl(\sigma_{f_1}^2\cdot\ldots\cdot\sigma_{f_k}^2\bigr)^{\frac{1}{2(k(d-1)+1)}}
\biggl|
\int_{\mathbb R^n}(\partial_\theta\varphi)(g)\,d\nu
\biggr|
\le
C_2(\eta,k,d,n)
t^{\frac{1}{k(d-1)+1}}
\biggl(\sum_{j=1}^k\sigma_{f_j}^2\biggr)^{\frac{k-1}{2(k(d-1)+1)}},
\]
where
\[
C_2(\eta,k,d,n)
=
C_2(k,d,n)\alpha_{n,d,k}(\eta)^{-\frac{1}{2(k(d-1)+1)}}
\]
and $C_2(k,d,n)$ is the constant from Lemma~\ref{lem-dimensional}.
Since
\[
\int_{\mathbb R^n}(\partial_\theta\varphi)(g)\,d\nu
=
\int_{\mathbb R^n}(\partial_\theta\varphi)(f)\,d\mu,
\]
the second estimate follows. 
This proves the lemma.
\end{proof}

The key ingredient in the proof of Theorems~\ref{th-main-lorentz} and~\ref{th-main-besov} is the following localization lemma
of Fradelizi and Gu\'edon, which allows one to upgrade dimension-dependent
results for polynomials and log-concave measures to dimension-free ones.

\begin{theorem}[see \cite{FrGue}]\label{loc-lem}
Let $K$ be a compact convex set in $\mathbb R^n$, let
$p\in\mathbb N$, and let
\[
u_j\colon K\to\mathbb R,
\quad j=1,\ldots,p,
\]
be continuous functions. Let $P_{u_1,\ldots,u_p}(K)\ne\emptyset$ be the set of all
log-concave measures $\nu$ supported in $K$ such that
\[
\int_K u_j\,d\nu\ge0,
\quad j=1,\ldots,p.
\]
Let
\[
\Phi\colon P(K)\to\mathbb R
\]
be a convex continuous functional, where $P(K)$ is the space of all Borel
probability measures on $K$, equipped with the weak topology. Then
\[
\sup_{\nu\in P_{u_1,\ldots,u_p}(K)} \Phi(\nu)
\]
is attained at a log-concave measure $\nu\in P_{u_1,\ldots,u_p}(K)$ such that
the affine span $S(\nu)$ of $\operatorname{supp}(\nu)$ satisfies
\[
\dim S(\nu)\le p.
\]
Consequently,
\[
\sup_{\mu\in P_{u_1,\ldots,u_p}(K)} \Phi(\nu)
=
\sup_{\substack{\nu\in P_{u_1,\ldots,u_p}(K)\\ \dim S(\nu)\le p}}
\Phi(\nu).
\]
\end{theorem}

We are now in a position to prove our main results.

\medskip

\noindent
{\bf Proof of Theorems \ref{th-main-lorentz} and \ref{th-main-besov}.}
Fix $\psi\in C_0^\infty(\mathbb{R}^k)$ with $0\le \psi\le 1$, $\theta\in \mathbb{R}^k$ with $|\theta|=1$, and $\varphi\in C_b^\infty(\mathbb{R}^k)$, satisfying
$\|\varphi\|_\infty\le t$, 
$\|\partial_\theta \varphi\|_\infty\le 1$.

{\bf Step 1.} 
First assume that $\mu$ is compactly supported. Let $K\subset\mathbb R^n$ be
a compact convex set such that
\[
\operatorname{supp}(\mu)\subset K.
\]

For the trace of $\Sigma_\mu^d(f)$, we have
\[
\operatorname{Tr}\Sigma_\mu^d(f)
=
\sum_{1\le |\mathbf m|\le d^{k-1}}
\operatorname{Var}_\mu(\hat f^{\mathbf m}).
\]
By the equivalence of polynomial moments with respect to log-concave measures \cite{Bobkov00},
for every $\mathbf m$ with $1\le |\mathbf m|\le d^{k-1}$,
\[
\operatorname{Var}_\mu(\hat f^{\mathbf m})
\le
\|\hat f^{\mathbf m}\|_{L^2(\mu)}^2
\le
\|\hat f_1\|_{L^{2|\mathbf m|}(\mu)}^{2m_1}
\ldots
\|\hat f_k\|_{L^{2|\mathbf m|}(\mu)}^{2m_k}
\le C_1(k,d),
\]
because
\[
\|\hat f_j\|_{L^2(\mu)}=1,
\quad \forall j\in\{1,\ldots,k\}.
\]
Therefore,
\[
\operatorname{Tr}\Sigma_\mu^d(f)\le C_2(k,d).
\]
Let
\[
M=M(k, d):=\dim \mathcal P_{d^{k-1}}(\mathbb R^k)-1.
\]
If $\lambda_1,\ldots,\lambda_M$ are the eigenvalues of $\Sigma_\mu^d(f)$ and
\[
\lambda_1=\lambda_{\min}(\Sigma_\mu^d(f)):=\min\{\lambda_1, \ldots, \lambda_M\},
\] 
then
\[
\det \Sigma_\mu^d(f)
=
\lambda_1\ldots\lambda_M
\le
\lambda_1\bigl(\operatorname{Tr}\Sigma_\mu^d(f)\bigr)^{M-1}\le \lambda_1\, \bigl(C_2(k,d)\bigr)^{M-1}.
\]
Thus,\[
\lambda_{\min}(\Sigma_\mu^d(f))
\ge
\eta\, \bigl(C_2(k,d)\bigr)^{-(M-1)}
=: \eta_0.
\]

Let
\[
\delta:=\frac{\eta_0}{4C_2(k,d)}.
\]
Let $\mathcal N_\delta\subset S^{M-1}$ be a $\delta$-net on the sphere of $\mathbb{R}^M$. We choose it so that
\[
|\mathcal N_\delta|
\le
\Bigl(1+\frac{2}{\delta}\Bigr)^M
=:C_3(\eta,k,d).
\]
For every $\xi\in\mathcal N_\delta$ we then have
\[
\operatorname{Var}_\mu
\bigl(\langle F_\mu^d(f),\xi\rangle\bigr)
=
\langle \Sigma_\mu^d(f)\xi,\xi\rangle
\ge \eta_0.
\]

Let
\[
I_j:=\int_{\mathbb R^n}f_j\,d\mu,
\qquad
\sigma_j^2:=\operatorname{Var}_\mu(f_j),
\qquad j=1,\ldots,k,
\]
\[
I_{\mathbf m}:=
\int_{\mathbb R^n}\hat f^{\mathbf m}\,d\mu,
\qquad
\sigma_{\mathbf m}^2:=
\operatorname{Var}_\mu(\hat f^{\mathbf m}),
\qquad
1\le|\mathbf m|\le d^{k-1},
\]
and
\[
I_\xi:=
\int_{\mathbb R^n}\langle F_\mu^d(f),\xi\rangle\,d\mu,
\qquad
\sigma_\xi^2:=
\operatorname{Var}_\mu
\bigl(\langle F_\mu^d(f),\xi\rangle\bigr),
\qquad
\xi\in\mathcal N_\delta.
\]

Let $P_u(K)$ be the set of all log-concave measures $\nu$
supported in $K$ satisfying the following inequalities:
\begin{equation}\label{variance-coincidence-1}
\int_K (f_j-I_j)\,d\nu\ge0,
\qquad
\int_K (I_j-f_j)\,d\nu\ge0,
\qquad
j=1,\ldots,k,
\end{equation}
\begin{equation}\label{variance-coincidence-2}
\int_K \bigl((f_j-I_j)^2-\sigma_j^2\bigr)\,d\nu\ge0,
\qquad
\int_K \bigl(\sigma_j^2-(f_j-I_j)^2\bigr)\,d\nu\ge0,
\qquad
j=1,\ldots,k,
\end{equation}
\begin{equation}\label{trace-coincidence-1}
\int_K (\hat f^{\mathbf m}-I_{\mathbf m})\,d\nu\ge0,
\qquad
\int_K (I_{\mathbf m}-\hat f^{\mathbf m})\,d\nu\ge0,
\qquad
1\le|\mathbf m|\le d^{k-1},
\end{equation}
\begin{equation}\label{trace-coincidence-2}
\int_K \bigl((\hat f^{\mathbf m}-I_{\mathbf m})^2-\sigma_{\mathbf m}^2\bigr)\,d\nu\ge0,
\qquad
\int_K \bigl(\sigma_{\mathbf m}^2-(\hat f^{\mathbf m}-I_{\mathbf m})^2\bigr)\,d\nu\ge0,
\qquad
1\le|\mathbf m|\le d^{k-1},
\end{equation}
\begin{equation}\label{nondeg-coincidence-1}
\int_K \bigl(\langle F_\mu^d(f),\xi\rangle-I_\xi\bigr)\,d\nu\ge0,
\qquad
\int_K \bigl(I_\xi-\langle F_\mu^d(f),\xi\rangle\bigr)\,d\nu\ge0,
\qquad
\xi\in\mathcal N_\delta,
\end{equation}
and
\begin{equation}\label{nondeg-coincidence-2}
\int_K
\bigl((\langle F_\mu^d(f),\xi\rangle-I_\xi)^2-\sigma_\xi^2\bigr)\,d\nu\ge0,
\qquad
\int_K
\bigl(\sigma_\xi^2-(\langle F_\mu^d(f),\xi\rangle-I_\xi)^2\bigr)\,d\nu\ge0,
\qquad
\xi\in\mathcal N_\delta.
\end{equation}
The total number of conditions is
\[
p:=p(\eta,k,d),
\]
depending only on $\eta,k,d$. Clearly, $\mu\in P_u(K)$.

\smallskip

{\bf Step 2.}
For every $\nu\in P_u(K)$, conditions \eqref{variance-coincidence-1} and
\eqref{variance-coincidence-2} imply
\[
\int_K f_j\,d\nu=I_j,
\qquad
\operatorname{Var}_\nu(f_j)=\sigma_j^2,
\qquad
j=1,\ldots,k.
\]
Hence the normalized components of $f$ with respect to $\nu$ coincide with
the functions $\hat f_j$ defined with respect to $\mu$. In particular,
\[
F_\nu^d(f)=F_\mu^d(f).
\]
Conditions \eqref{trace-coincidence-1} and \eqref{trace-coincidence-2} imply
\[
\operatorname{Tr}\Sigma_\nu^d(f)=\operatorname{Tr}\Sigma_\mu^d(f)\le C_2(k, d).
\]
Finally, conditions \eqref{nondeg-coincidence-1} and
\eqref{nondeg-coincidence-2} imply that, for every $\xi\in\mathcal N_\delta$,
\[
\langle \Sigma_\nu^d(f)\xi,\xi\rangle
=
\operatorname{Var}_\nu
\bigl(\langle F_\nu^d(f),\xi\rangle\bigr)
=
\operatorname{Var}_\mu
\bigl(\langle F_\mu^d(f),\xi\rangle\bigr)
\ge\eta_0.
\]

We now show that $\Sigma_\nu^d(f)$ is uniformly nondegenerate for every
$\nu\in P_u(K)$. Let $\zeta\in S^{M-1}$. Choose $\xi\in\mathcal N_\delta$
such that
\[
|\zeta-\xi|\le\delta.
\]
Then
\[
\langle \Sigma_\nu^d(f)\zeta,\zeta\rangle
=
\langle \Sigma_\nu^d(f)\xi,\xi\rangle
+
2\langle \Sigma_\nu^d(f)\xi,\zeta-\xi\rangle
+
\langle \Sigma_\nu^d(f)(\zeta-\xi),\zeta-\xi\rangle.
\]
Since $\Sigma_\nu^d(f)$ is nonnegative definite,
\[
\langle \Sigma_\nu^d(f)(\zeta-\xi),\zeta-\xi\rangle\ge0.
\]
Moreover,
\[
|\langle \Sigma_\nu^d(f)\xi,\zeta-\xi\rangle|
\le
\|\Sigma_\nu^d(f)\|_{\operatorname{op}}|\zeta-\xi|
\le
\delta\|\Sigma_\nu^d(f)\|_{\operatorname{op}}
\le \delta \operatorname{Tr}\Sigma_\nu^d(f)\le \delta\, C_2(k, d) = \frac{\eta_0}{4}.
\]
Therefore,
\[
\langle \Sigma_\nu^d(f)\zeta,\zeta\rangle
\ge
\langle \Sigma_\nu^d(f)\xi,\xi\rangle
- \frac{\eta_0}{2}\ge \frac{\eta_0}{2}.
\]
Thus,
\[
\lambda_{\min}(\Sigma_\nu^d(f))\ge\frac{\eta_0}{2},
\]
and therefore
\[
\det \Sigma_\nu^d(f)\ge
\Bigl(\frac{\eta_0}{2}\Bigr)^M
=:C_4(\eta,k,d)>0
\quad
\forall \nu\in P_u(K).
\]
\smallskip
{\bf Step 3.}
Let
\[
\Phi_1(\nu)=\int_K\psi(f)\,d\nu
\]
and
\[
\Phi_2(\nu)
=
\biggl|
\int_K(\partial_\theta\varphi)(f)\,d\nu
\biggr|.
\]
Both of them are convex continuous functionals on $P(K)$. Since $\mu\in P_u(K)$,
Theorem~\ref{loc-lem} gives
\begin{equation}\label{eq-comparison-1}
\int_K\psi(f)\,d\mu
\le
\sup_{\substack{\nu\in P_u(K)\\ \dim S(\nu)\le p}}
\int_K\psi(f)\,d\nu
\end{equation}
and
\begin{equation}\label{eq-comparison-2}
\biggl|
\int_{\mathbb R^n}(\partial_\theta\varphi)(f)\,d\mu
\biggr|
\le
\sup_{\substack{\nu\in P_u(K)\\ \dim S(\nu)\le p}}
\biggl|
\int_K(\partial_\theta\varphi)(f)\,d\nu
\biggr|.
\end{equation}
Let $\nu\in P_u(K)$ be such that $\dim S(\nu)\le p$. 
Since $\nu$, viewed as a measure on the affine span $S(\nu)$ of its support,
is absolutely continuous with respect to the Lebesgue measure on $S(\nu)$, and
since
\[
\det \Sigma_\nu^d(f)\ge C_4(\eta,k,d)>0,
\]
Lemma~\ref{lem-n-less-k}, applied on $S(\nu)$, implies
\[
\dim S(\nu)\ge k.
\]
Let
\[
C_5(\eta, k, d):=\max\bigl\{C\bigl(C_4(\eta, k, d),k,d, j\bigr)\colon k\le j\le p\bigr\},
\]
where 
$C(\cdot, k, d, j)$ is the constant from Lemma
\ref{lem-dimensional-2}.
Viewing $\nu$ as an absolutely continuous log-concave measure on $S(\nu)$ and applying
Lemma~\ref{lem-dimensional-2} in dimension $\dim S(\nu)$, we obtain
\[
\int_K\psi(f)\,d\nu
\le
C_5(\eta,k,d)
a^{-\frac{k}{2(k(d-1)+1)}}
\|\psi\|_{L^1(\mathbb R^k)}^{\frac{1}{k(d-1)+1}}
\]
and
\[
\biggl|
\int_K(\partial_\theta\varphi)(f)\,d\nu
\biggr|
\le
C_5(\eta,k,d)k^{\frac{k-1}{2(k(d-1)+1)}}
\bigl(a^{-k}b^{k-1}\bigr)^{\frac{1}{2(k(d-1)+1)}}
t^{\frac{1}{k(d-1)+1}}.
\]
By \eqref{eq-comparison-1} and \eqref{eq-comparison-2}, both estimates are also valid for the initial measure $\mu$.
This proves the compactly supported case.

\smallskip

{\bf Step 4.}
Let now $\mu$ be an arbitrary log-concave measure on $\mathbb R^n$. Let
$B_N$ be the Euclidean ball of radius $N$ centered at the origin, and define
\[
\mu_N:=\bigl(\mu(B_N)\bigr)^{-1}I_{B_N}\cdot \mu.
\]
For all sufficiently large $N$, this measure is well defined. Since $B_N$ is
convex, $\mu_N$ is again log-concave and compactly supported.
Moreover, for every $h\in L^1(\mu)$,
\[
\int_{\mathbb R^n}h\,d\mu_N
=
\frac{1}{\mu(B_N)}\int_{B_N}h\,d\mu
\to
\int_{\mathbb R^n}h\,d\mu.
\]
In particular, 
\[
\operatorname{Var}_{\mu_N}(f_j)\to
\operatorname{Var}_{\mu}(f_j)
\quad \forall j\in\{1,\ldots,k\},
\]
and
\[
\Sigma_{\mu_N}^d(f)\to \Sigma_\mu^d(f)
\]
elementwise. Consequently,
\[
\det \Sigma_{\mu_N}^d(f)\to \det \Sigma_\mu^d(f).
\]
Therefore, there exists $N_0$ such that, for all $N\ge N_0$,
\[
\frac{a}{2}\le \operatorname{Var}_{\mu_N}(f_j)\le 2b
\quad \forall j\in\{1,\ldots,k\},
\]
and
\[
\det \Sigma_{\mu_N}^d(f)\ge \frac{\eta}{2}.
\]

By the already proved compactly supported case, applied to $\mu_N$ with
parameters $a/2$, $2b$, and $\eta/2$, we obtain
\[
\int_{\mathbb{R}^n}\psi(f)\,d\mu_N
\le
C_6(\eta,k,d)
a^{-\frac{k}{2(k(d-1)+1)}}
\|\psi\|_{L^1(\mathbb R^k)}^{\frac{1}{k(d-1)+1}}
\]
and
\[
\biggl|
\int_{\mathbb R^n}(\partial_\theta\varphi)(f)\,d\mu_N
\biggr|
\le
C_6(\eta,k,d)
\bigl(a^{-k}b^{k-1}\bigr)^{\frac{1}{2(k(d-1)+1)}}
t^{\frac{1}{k(d-1)+1}}
\]
for every $N\ge N_0$. Finally, since $\psi(f)$ and
$(\partial_\theta\varphi)(f)$ are bounded, we may pass to the limit as
$N\to\infty$ and obtain
\begin{equation}\label{eq-resulting-1}
\int_{\mathbb{R}^n}\psi(f)\,d\mu
\le
C_6(\eta,k,d)
a^{-\frac{k}{2(k(d-1)+1)}}
\|\psi\|_{L^1(\mathbb R^k)}^{\frac{1}{k(d-1)+1}}
\end{equation}
and
\begin{equation}\label{eq-resulting-2}
\biggl|
\int_{\mathbb R^n}(\partial_\theta\varphi)(f)\,d\mu
\biggr|
\le
C_6(\eta,k,d)
\bigl(a^{-k}b^{k-1}\bigr)^{\frac{1}{2(k(d-1)+1)}}
t^{\frac{1}{k(d-1)+1}}.
\end{equation}

\smallskip

{\bf Step 5.}
Since $\psi\in C_0^\infty(\mathbb R^k)$ with $0\le \psi\le 1$ was arbitrary,
\eqref{eq-resulting-1} and Lemma~\ref{lem-lorentz-equivalence} imply
\[
\mu(f\in A)
\le
C_6(\eta,k,d)
a^{-\frac{k}{2(k(d-1)+1)}}
\bigl(\lambda_k(A)\bigr)^{\frac{1}{k(d-1)+1}}
\]
for every Borel set $A\subset \mathbb R^k$.
Similarly, since $\varphi\in C_b^\infty(\mathbb R^k)$ and $\theta$ were
arbitrary, by the definition of $\sigma(\mu\circ f^{-1},t)$,
\eqref{eq-resulting-2} implies
\[
\sigma(\mu\circ f^{-1},t)
\le
C_6(\eta,k,d)
\bigl(a^{-k}b^{k-1}\bigr)^{\frac{1}{2(k(d-1)+1)}}
t^{\frac{1}{k(d-1)+1}}.
\]
This proves Theorems \ref{th-main-lorentz} and \ref{th-main-besov}.
\qed

\section{Probability applications}
\label{sect-probab}

We begin with the following extension of \cite[Lemma~2.4]{NP}.

\begin{lemma}\label{lem-moments-lc-closure}
Let $X=(X_1,\ldots,X_k)\in\mathcal P_{d,k}^{\rm lc}$. Then $X$ has finite
moments of all orders and $\Sigma^d(X)$ is well defined. Moreover, if
$X_m=(X_{m,1},\ldots,X_{m,k})\in\mathcal P_{d,k}^{\rm lc}$ converges in
distribution to $X$, then, for every $p\ge1$,
\[
\max_{1\le j\le k}\sup_{m\ge1}
\mathbb E\bigl[|X_{m,j}|^p\bigr]<\infty.
\]
\end{lemma}

\begin{proof}
Let
$X_m\in\mathcal P_{d,k}^{\rm lc}$ converge in distribution to $X$. Fix
$j\in\{1,\ldots,k\}$ and $p\ge1$. Since $X_{m,j}$ converges in distribution
to $X_j$, the sequence $X_{m,j}$ is tight. Hence there exists $R>0$ such that
\[
\mathbb P\bigl(|X_{m,j}|<R\bigr)\ge \frac34
\quad
\forall m\in\mathbb N.
\]
For every $m$, by the definition of $\mathcal P_{d,k}^{\rm lc}$, there exists
a sequence
\[
Y_{\ell}^m=f_{m,\ell}(\xi_{m,\ell}),\quad \ell\in\mathbb N,
\]
converging in distribution to $X_m$, where
$f_{m,\ell}\in\mathcal P_d(\mathbb R^{n_{m,\ell}};\mathbb R^k)$ and
$\xi_{m,\ell}$ has a log-concave distribution. Since the set
$\{x\in\mathbb R\colon |x|<R\}$ is open, the Portmanteau theorem, see
\cite[Theorem~2.1]{Billingsley99}, gives
\[
\liminf_{\ell\to\infty}
\mathbb P\bigl(|Y^m_{\ell,j}|<R\bigr)
\ge
\mathbb P\bigl(|X_{m,j}|<R\bigr)
\ge \frac34.
\]
Therefore, for every $m$ there exists $\ell_m$ such that
\[
\mathbb P\bigl(|Y^m_{\ell,j}|<R\bigr)\ge \frac12
\quad
\forall \ell\ge \ell_m.
\]
By the Carbery--Wright inequality \eqref{CW-est},
\[
\bigl(\mathbb E\bigl[|Y^m_{\ell,j}|^2\bigr]\bigr)^{1/(2d)}
\mathbb P\bigl(|Y^m_{\ell,j}|<R\bigr)
\le
C d R^{1/d}
\quad
\forall \ell\ge \ell_m.
\]
Thus,
\[
\sup_{m\ge1}\sup_{\ell\ge \ell_m}
\mathbb E\bigl[|Y^m_{\ell,j}|^2\bigr]<\infty.
\]
The equivalence of all moments for polynomials of fixed degree with respect
to log-concave measures gives, for every $p\ge1$,
\[
\sup_{m\ge1}\sup_{\ell\ge \ell_m}
\mathbb E\bigl[|Y^m_{\ell,j}|^p\bigr]<\infty.
\]
For $N\in\mathbb N$, set
\[
\psi_N(x)=\min\{|x|^p,N\}.
\]
Since $\psi_N$ is bounded and continuous and $Y^m_\ell$ converges in
distribution to $X_m$ as $\ell\to\infty$, we have
\[
\mathbb E\bigl[\psi_N(X_{m,j})\bigr]
=
\lim_{\ell\to\infty}\mathbb E\bigl[\psi_N(Y^m_{\ell,j})\bigr]
\le 
\sup_{r\ge1}\sup_{\ell\ge \ell_r}
\mathbb E\bigl[|Y^r_{\ell,j}|^p\bigr].
\]
Letting $N\to\infty$ and using the monotone convergence theorem, we get
\[
\sup_{m\ge1}\mathbb E\bigl[|X_{m,j}|^p\bigr]
\le 
\sup_{r\ge1}\sup_{\ell\ge \ell_r}
\mathbb E\bigl[|Y^r_{\ell,j}|^p\bigr]<\infty.
\]
Since $j$ was arbitrary,
\[
\max_{1\le j\le k}\sup_{m\ge1}
\mathbb E\bigl[|X_{m,j}|^p\bigr]<\infty.
\]
Moreover, for every $N\in\mathbb N$,
\[
\mathbb E\bigl[\psi_N(X_j)\bigr]
=
\lim_{m\to\infty}\mathbb E\bigl[\psi_N(X_{m,j})\bigr]
\le 
\sup_{m\ge1}\mathbb E\bigl[|X_{m,j}|^p\bigr]<\infty.
\]
Again, letting $N\to\infty$ and using the monotone convergence theorem, we get
\[
\mathbb E\bigl[|X_j|^p\bigr]<\infty.
\]
Thus, every component of $X$ has finite moments of all orders, and hence the matrix
$\Sigma^d(X)$ is well defined.
\end{proof}

\medskip

\noindent
{\bf Proof of Corollary~\ref{cor-main}.}
Choose a sequence $X_m=f_m(\xi_m)$ converging in distribution to $X$, where
$f_m\in\mathcal P_d(\mathbb R^{n_m};\mathbb R^k)$ and $\xi_m$ has a
log-concave distribution. By Lemma~\ref{lem-moments-lc-closure}, for every
$p\ge1$,
\[
\max_{1\le j\le k}\sup_{m\ge1}
\mathbb E\bigl[|X_{m,j}|^p\bigr]<\infty.
\]
Hence, by \cite[Theorem~3.5]{Billingsley99}, for every polynomial $Q$ on
$\mathbb R^k$,
\[
\mathbb E\bigl[Q(X_m)\bigr]\to \mathbb E\bigl[Q(X)\bigr].
\]
Applying this to the polynomial moments entering $\operatorname{Var}$ and
$\Sigma^d$, we get
\[
\operatorname{Var}(X_{m,j})\to \operatorname{Var}(X_j)
\quad
\forall j\in\{1,\ldots,k\},
\]
and
\[
\Sigma^d(X_m)\to \Sigma^d(X).
\]
Therefore, for all sufficiently large $m$,
\[
\frac a2\le \operatorname{Var}(X_{m,j})\le 2b
\quad
\forall j\in\{1,\ldots,k\},
\quad\text{and}\quad
\det \Sigma^d(X_m)\ge \frac{\eta}{2}.
\]

By Theorems~\ref{th-main-lorentz} and~\ref{th-main-besov}, there exist
constants 
\[
C_1:=C_1(a,\eta,k,d)>0\quad\text{and}\quad
C_2:=C_2(a,b,\eta,k,d)>0
\] 
such that, for all sufficiently large $m$,
\[
\mathbb E\bigl[\psi(X_m)\bigr]
\le C_1\|\psi\|_{L^1(\mathbb R^k)}^\frac{1}{k(d-1)+1}
\]
for every $\psi\in C_0^\infty(\mathbb R^k)$ satisfying $0\le\psi\le1$, and
\[
\left|\mathbb E\bigl[\partial_\theta\varphi(X_m)\bigr]\right|
\le C_2t^\frac{1}{k(d-1)+1}
\]
for every unit vector $\theta\in\mathbb R^k$, every $t>0$, and every
$\varphi\in C_b^\infty(\mathbb R^k)$ satisfying
\[
\|\varphi\|_\infty\le t,
\quad
\|\partial_\theta\varphi\|_\infty\le 1.
\]

Passing to the limit as $m\to\infty$, we obtain the same two estimates with
$X$ in place of $X_m$. 
Lemma~\ref{lem-lorentz-equivalence} gives
\[
\mathbb P(X\in A)
\le
C_1\bigl(\lambda_k(A)\bigr)^\frac{1}{k(d-1)+1}
\]
for every Borel set $A\subset\mathbb R^k$, and the definition of
$\sigma(\mathbb P\circ X^{-1},\cdot)$ gives
\[
\sigma(\mathbb P\circ X^{-1},t)
\le
C_2t^\frac{1}{k(d-1)+1}
\quad
\forall t>0.
\]
This completes the proof.
\qed

\medskip

\noindent
{\bf Proof of Corollary~\ref{cor-converg}.}
By Corollary~\ref{cor-main}, there exists
$C_0:=C_0(a,b,\eta,k,d)>0$ such that
\[
\sigma(\mathbb P\circ X^{-1},t)\le C_0t^\frac{1}{k(d-1)+1},
\quad
\sigma(\mathbb P\circ Y^{-1},t)\le C_0t^\frac{1}{k(d-1)+1}
\quad
\forall t>0.
\]
By~\cite[Lemma 3.1]{Kos-FCAA},
\[
d_{\rm TV}(X,Y)
\le
6\sqrt{k}\max\bigl\{\sigma(\mathbb P\circ X^{-1},t),
\sigma(\mathbb P\circ Y^{-1},t)\bigr\}
+\sqrt{k}\,t^{-1}d_{\rm KR}(X,Y)
\quad
\forall t\in(0,1].
\]
Hence,
\[
d_{\rm TV}(X,Y)
\le
C_1t^\frac{1}{k(d-1)+1}+\sqrt{k}t^{-1}d_{\rm KR}(X,Y)
\quad
\forall t\in(0, 1],
\]
where $C_1:=C_1(a,b,\eta,k,d)$ depend only on $a,b,\eta,k,d$. Choosing
\[
t=\Bigl(\frac{1}{2}d_{\rm KR}(X,Y)\Bigr)^{\frac{k(d-1)+1}{k(d-1)+2}}
\]
gives
\[
d_{\rm TV}(X,Y)
\le
C_2\, d_{\rm KR}(X,Y)^{\frac{1}{k(d-1)+2}}
\]
for some $C_2:=C_2(a,b,\eta,k,d)$ depend only on $a,b,\eta,k,d$.

It remains to prove the second assertion. Let $X_n\in\mathcal P_{d,k}^{\rm lc}$
converge in distribution to $X_\infty$, and assume that
$\mathbb P\circ X_\infty^{-1}$ is absolutely continuous. By the definition of
$\mathcal P_{d,k}^{\rm lc}$ and a diagonal argument,
$X_\infty\in\mathcal P_{d,k}^{\rm lc}$. 
Since the distribution is absolutely continuous, 
$\operatorname{Var}(X_{\infty,j})>0$
for all $j=1,\ldots,k$,
and
\[
\det\Sigma^d(X_\infty)>0.
\]
By Lemma~\ref{lem-moments-lc-closure} and \cite[Theorem~3.5]{Billingsley99},
\[
\operatorname{Var}(X_{n,j})\to \operatorname{Var}(X_{\infty,j})
\quad
\forall j\in\{1,\ldots,k\},
\]
and
\[
\Sigma^d(X_n)\to \Sigma^d(X_\infty).
\]
Hence there exist $a,b,\eta>0$ and $n_0\in\mathbb N$ such that, for all
$n\ge n_0$,
\[
a\le \operatorname{Var}(X_{n,j})\le b,
\quad
a\le \operatorname{Var}(X_{\infty,j})\le b
\quad
\forall j\in\{1,\ldots,k\},
\]
and
\[
\det\Sigma^d(X_n)\ge\eta,
\quad
\det\Sigma^d(X_\infty)\ge\eta.
\]
Applying the first part of the corollary to $X_n$ and $X_\infty$, we obtain
\[
d_{\rm TV}(X_n,X_\infty)
\le
C d_{\rm KR}(X_n,X_\infty)^{\frac{1}{k(d-1)+2}}
\quad
\forall n\ge n_0.
\]
Increasing $C$ to account for the finitely many indices $n<n_0$ completes the
proof.
\qed

\section*{Use of AI Tools}

ChatGPT was used for language editing, stylistic suggestions, draft wording for
selected passages, and help with locating some references. All AI-generated
text and suggested references were checked, corrected where necessary, and
substantially revised by the authors. The authors take full responsibility for
the content of the paper.





\section*{Acknowledgements}

The authors would like to thank Sergey Tikhonov for reading the manuscript and for valuable comments and suggestions.

\smallskip

This work was supported by the AEI grants 
RYC2023-043616-I and
PID2025-169712NA-I00 funded by MICIU/AEI/10.13039/501100011033,
and by the Spanish State Research Agency, through the Severo Ochoa and Mar\'ia de Maeztu Program for Centers and
Units of Excellence in R\&D (CEX2020-001084-M).
The authors thanks CERCA Programme (Generalitat de Catalunya) for institutional support.
{\sloppy
	
}

\end{document}